\documentclass[12pt,fleqn,draft]{article} 
\usepackage{amsmath,amssymb,amsthm,amsfonts,bm}
\usepackage{enumerate}
\usepackage{indentfirst}
\usepackage{cite}
\usepackage[usenames,dvipsnames]{color}

\makeatletter
\def\@cite#1#2{[{{\bfseries #1}\if@tempswa , #2\fi}]}
\renewcommand{\section}{%
\@startsection{section}{1}{\z@}
{0.5truecm plus -1ex minus -.2ex}%
{1.0ex plus .2ex}{\bfseries\large}}
\def\@seccntformat#1{\csname the#1\endcsname.\ }
\makeatother

\numberwithin{equation}{section} 
\theoremstyle{theorem}
\newtheorem{thm}{Theorem}[section]

\newtheorem{lem}[thm]{Lemma}

\theoremstyle{definition}
\newtheorem{df}{Definition}[section]
\newtheorem{remark}{Remark}[section]

\newtheorem*{prth1.1}{Proof of Theorem 1.1}
\newtheorem*{prth1.2}{Proof of Theorem 1.2}
\newtheorem*{prlem2.1}{Proof of Lemma 2.1}
\newtheorem*{prlem2.2}{Proof of Lemma 2.2}

\newcommand{\ep}{\varepsilon}

\let\hat\widehat

\def\Pi{\hat\pi}

\begin{document}
\footnote[0]
    {2010 {\it Mathematics Subject Classification}\/: 
    35G31, 35B25, 35D30, 35A40.     
    }
\footnote[0] 
    {{\it Key words and phrases}\/: 
    Cahn--Hilliard approaches; nonlinear diffusions; 
    parabolic-elliptic chemotaxis systems; 
    existence; time discretizations.  
} 
%==========================title==========================
\begin{center}
    \Large{{\bf Convergence of a Cahn--Hilliard type system to \\ 
a parabolic-elliptic chemotaxis system \\ 
with nonlinear diffusion}}
\end{center}
\vspace{5pt}
%===========================author=========================
\begin{center}
    Shunsuke Kurima%
   %\footnote{Corresponding author}
    \\
    \vspace{2pt}
    Department of Mathematics, 
    Tokyo University of Science\\
    1-3, Kagurazaka, Shinjuku-ku, Tokyo 162-8601, Japan\\
    {\tt shunsuke.kurima@gmail.com}\\
    \vspace{2pt}
\end{center}
\begin{center}    
    \small \today
\end{center}

\vspace{2pt}
%=====================  Abstract  =======================
\newenvironment{summary}
{\vspace{.5\baselineskip}\begin{list}{}{%
     \setlength{\baselineskip}{0.85\baselineskip}
     \setlength{\topsep}{0pt}
     \setlength{\leftmargin}{12mm}
     \setlength{\rightmargin}{12mm}
     \setlength{\listparindent}{0mm}
     \setlength{\itemindent}{\listparindent}
     \setlength{\parsep}{0pt}
     \item\relax}}{\end{list}\vspace{.5\baselineskip}}
\begin{summary}
{\footnotesize {\bf Abstract.} 
This paper deals with a parabolic-elliptic chemotaxis system with nonlinear diffusion. 
It was proved that there exists a solution of 
a Cahn--Hilliard system as an approximation of 
a nonlinear diffusion equation by applying an abstract theory 
by Colli--Visintin [Comm.\ Partial Differential Equations {\bf 15} (1990), 737--756] 
for a doubly nonlinear evolution inclusion 
with some bounded monotone operator and 
subdifferential operator 
of a proper lower semicontinuous convex function 
(cf.\ Colli--Fukao [J. Math.\ Anal.\ Appl.\ {\bf 429} (2015), 1190--1213]). 
Moreover, Colli--Fukao [J. Differential Equations {\bf 260} (2016), 6930--6959]  
established existence of solutions to the nonlinear diffusion equation 
by passing to the limit in the Cahn--Hilliard equation. 
However, Cahn--Hilliard approaches to chemotaxis systems with
nonlinear diffusions seem not to be studied yet. 
This paper will try to derive existence of
solutions to a parabolic-elliptic chemotaxis system with nonlinear diffusion 
by passing to the limit in a Cahn--Hilliard type chemotaxis system. 
}
\end{summary}
\vspace{10pt}

\newpage
%%==============================================================%%
%%==============                                  ==============%%
%%======                      Section1                    ======%%
%%====                                                      ====%%
%%==                                                         ==%%
%%====               Introduction Results            ====%%
%%======                                                  ======%%
%%==============                                  ==============%%
%%==============================================================%%

\section{Introduction}\label{Sec1}

A relation between a nonlinear diffusion equation  
and a Cahn--Hilliard equation 
has been studied.    
Colli--Fukao \cite{CF-2015, CF-2016} 
considered the nonlinear diffusion equation 
\begin{align}\label{E}\tag*{(E)}
u_{t} - \Delta\xi = g,\ \xi \in \beta(u) 
\quad \mbox{in}\ 
\Omega\times(0, T)  
\end{align}
and the Cahn--Hilliard type of approximate equation    
\begin{align}\label{Eep}\tag*{(E)$_{\ep}$}
\begin{cases}
(u_{\ep})_{t} - \Delta\mu_{\ep} = 0
& \mbox{in}\ 
\Omega\times(0, T), 
\\
\mu_{\ep} = -\ep\Delta u_{\ep} + \xi_{\ep} + \pi_{\ep}(u_{\ep}) - f,\ 
\xi_{\ep} \in \beta(u_{\ep})
& \mbox{in}\ 
\Omega\times(0, T), 
\end{cases}
\end{align}
where $\Omega \subset \mathbb{R}^d$ ($d = 1, 2, 3$) is a bounded domain, 
$\ep \in (0, 1]$, 
$T>0$, 
$\beta : \mathbb{R}\to\mathbb{R}$ is a multi-valued maximal monotone function, 
$\pi_{\ep} : \mathbb{R}\to\mathbb{R}$ is an anti-monotone function which goes to $0$ 
in some sense as $\ep \searrow 0$, 
$f : \Omega\times(0, T) \to \mathbb{R}$ is a given function. 
To prove existence of solutions to \ref{Eep} they used one more approximation 
\begin{align}\label{Eeplamb}\tag*{(E)$_{\ep, \lambda}$}
\begin{cases}
(u_{\ep, \lambda})_{t} -\Delta\mu_{\ep, \lambda} = 0
& \mbox{in}\ \Omega\times(0, T), 
\\
\mu_{\ep, \lambda} = \lambda(u_{\ep, \lambda})_{t}  
           -\ep\Delta u_{\ep, \lambda} + \beta_{\lambda}(u_{\ep, \lambda}) 
           + \pi_{\ep}(u_{\ep, \lambda}) - f 
& \mbox{in}\ \Omega\times(0, T), 
\end{cases}
\end{align}
where $\lambda>0$ and 
$\beta_{\lambda}$  
is the Yosida approximation of $\beta$ on $\mathbb{R}$. 
In \cite{CF-2015} they proved 
existence of solutions to \ref{Eeplamb} 
by applying an abstract theory 
by Colli--Visintin \cite{CV-1990} 
for the doubly nonlinear evolution inclusion:  
    \[
    Au'(t) + \partial\psi(u(t)) \ni k(t)
    \]
with some bounded monotone operator $A$ and 
subdifferential operator $\partial\psi$ 
of a proper lower semicontinuous convex function $\psi$.  
Next, in \cite{CF-2016} they
derived existence of solutions to \ref{Eep} 
and \ref{E} by passing to the limit in \ref{Eeplamb} as $\lambda \searrow 0$ 
and in \ref{Eep} as $\ep \searrow 0$ individually. 
On the other hand, 
relations between chemotaxis systems with nonlinear diffusions  
and Cahn--Hilliard type chemotaxis systems seem not be studied yet. 

In this paper we consider 
the parabolic-elliptic chemotaxis system with nonlinear diffusion 
\begin{equation}\label{P}\tag*{(P)}
\begin{cases}
u_t - \Delta \beta(u) + \eta\nabla\cdot(u\nabla v) = g  
&\mbox{in}\ \Omega \times (0, T), \\[2mm] 
0 = \Delta v - v +u 
&\mbox{in}\ \Omega \times (0, T), \\[2mm]
\partial_\nu \beta(u) = \partial_\nu v = 0 
&\mbox{on}\ \partial\Omega \times (0, T), \\[2mm] 
u(0) = u_0 &\mbox{in}\ \Omega,   
\end{cases}
\end{equation}
where $\Omega \subset \mathbb{R}^d$ ($d= 1, 2, 3$) is a bounded domain 
with smooth boundary $\partial\Omega$, $T>0$, $\eta \in \mathbb{R}$, 
$\beta : \mathbb{R} \to \mathbb{R}$ is a single-valued maximal monotone function, 
$\partial_\nu$ denotes differentiation with respect to
the outward normal of $\partial\Omega$, 
$g : \Omega\times(0, T) \to \mathbb{R}$ and $u_{0} : \Omega \to \mathbb{R}$ 
are given functions. 
Also, in reference to \cite{CF-2016}, we deal with the Cahn--Hilliard type 
chemotaxis system  
\begin{equation}\label{Pep}\tag*{(P)$_{\ep}$}
\begin{cases}
(u_{\ep})_t - \Delta \mu_{\ep} + \eta\nabla\cdot(u_{\ep} \nabla v_{\ep}) = 0 
&\mbox{in}\ \Omega \times (0, T), \\[2mm] 
\mu_{\ep} = - \ep\Delta u_{\ep} + \beta(u_{\ep}) + \pi_{\ep}(u_{\ep}) - f 
&\mbox{in}\ \Omega \times (0, T), \\[2mm]
0 = \Delta v_{\ep} - v_{\ep} + u_{\ep} 
&\mbox{in}\ \Omega \times (0, T), \\[2mm]
\partial_{\nu} \mu_{\ep} = \partial_{\nu} u_{\ep} = \partial_{\nu} v_{\ep} = 0 
&\mbox{on}\ \partial\Omega \times (0, T), \\[2mm] 
u_{\ep} (0) = u_{0\ep} &\mbox{in}\ \Omega,  
\end{cases}
\end{equation}
where $\ep \in (0, 1]$, 
$\pi_{\ep} : \mathbb{R} \to \mathbb{R}$ is an anti-monotone function,  
$f : \Omega\times(0, T) \to \mathbb{R}$ and  
$u_{0\ep} : \Omega \to \mathbb{R}$ are given functions. 
Moreover, in reference to \cite{CF-2016}, in this paper 
we assume that 
\begin{enumerate} 
\setlength{\itemsep}{2mm}
\item[(A1)] 
$\beta : \mathbb{R} \to \mathbb{R}$                                
is a single-valued maximal monotone function with effective domain $D(\beta)$  
such that 
there exists a proper lower semicontinuous convex function 
$\hat{\beta} : \mathbb{R} \to [0, +\infty]$ 
with effective domain $D(\hat{\beta})$ 
satisfying that  
$\hat{\beta}(0) = 0$ and 
$\beta = \partial\hat{\beta}$, 
where $\partial\hat{\beta}$  
is the subdifferential of $\hat{\beta}$.
\item[(A2)] There exist constants $c_1, c_2 > 0$ such that 
\[
\hat{\beta}(r) \geq c_1 |r|^4 - c_2 
\quad \mbox{for all}\ r \in \mathbb{R}. 
\] 
\item[(A3)] $g\in L^2\bigl(0, T; L^2(\Omega)\bigr)$ 
and $\int_{\Omega}g(t)\,dx = 0$ for a.a.\ $t \in (0, T)$. 
Then we fix a solution 
$f\in L^2\bigl(0, T; H^2(\Omega)\bigr)$ of 
\begin{equation*}%\label{a3eq}
    \begin{cases}
        -\Delta f(t) = g(t) 
        & \mbox{a.e.\ in}\ \Omega,
    \\[2mm] 
        \partial_{\nu}f(t) = 0 
        & \mbox{in the sense of traces on}\ \partial\Omega 
    \end{cases}
\end{equation*}
for a.a.\ $t\in(0, T)$, that is, 
\begin{equation*}
\int_{\Omega}\nabla f(t)\cdot\nabla z = \int_{\Omega}g(t)z \quad 
\mbox{for all}\ z\in H^1(\Omega). 
\end{equation*}
\item[(A4)] $\pi_{\ep} : \mathbb{R} \to \mathbb{R}$ is                          
a Lipschitz continuous function 
and there exists a constant $c_{3}>0$ such that 
\begin{equation*}%\label{a2ineq}
|\pi_{\ep}(0)| + \|\pi_{\ep}'\|_{L^{\infty}(\mathbb{R})} \leq c_{3}\ep 
\quad \mbox{for all}\ \ep\in(0, 1].
\end{equation*} 
\item[(A5)] 
$u_{0} : \Omega \to \mathbb{R}$ is a measurable function 
satisfying $\hat{\beta}(u_{0})\in L^1(\Omega)$ 
and $m_{0} := \frac{1}{|\Omega|}\int_{\Omega} u_{0} \in \mbox{Int}\, D(\beta)$. 
Moreover, 
let $u_{0\ep}\in H^2(\Omega)$  
fulfill $\hat{\beta}(u_{0\ep})\in L^1(\Omega)$, 
$\frac{1}{|\Omega|}\int_{\Omega}u_{0\ep} = m_{0}$  
and 
\begin{equation*}%\label{a4ineq}
\|u_{0\ep}\|_{L^4(\Omega)}^4 \leq c_{4},\quad 
\int_{\Omega}\hat{\beta}(u_{0\ep}) \leq c_{4}, \quad 
\ep\|\nabla u_{0\ep}\|_{L^2(\Omega)}^2 \leq c_{4}
\end{equation*}
for all $\ep\in(0, 1]$, where $c_{4} > 0$  
is a constant independent of $\ep$; 
in addition, $u_{0\ep}\to u_{0}$ in $L^4(\Omega)$ as $\ep\searrow0$.
\end{enumerate}
\smallskip

\noindent The three functions 
\begin{align*}
&\beta_{1}(r) = r|r|^{m-1},\ r \in \mathbb{R}, 
\quad 
\beta_{2}(r) = \ln \frac{1+r}{1-r},\ r \in (-1, 1), 
\\ 
&\beta_{3}(r) = |r|\ln \frac{1+r}{1-r},\ r \in (-1, 1), 
\end{align*}
where $m\geq3$ is some constant, are examples of $\beta$. 
Indeed, we have that 
\[
\ln \frac{1+r}{1-r} \geq \frac{8}{3}r^3 \quad \mbox{for all}\ r \in (-1, 1). 
\]
The function $\beta_{1}$ appears in the porous media equation 
(see e.g., \cite{ASS-2016, M-2010, V-2007, Y-2008}).  

%-----------------------------------------------------------------
%
%                      Notations
%
%-----------------------------------------------------------------
%
Let us define the Hilbert spaces 
   $$
   H:=L^2(\Omega), \quad V:=H^1(\Omega)
   $$
 with inner products 
 \begin{align*} 
 &(u_{1}, u_{2})_{H}:=\int_{\Omega}u_{1}u_{2}\,dx \quad  (u_{1}, u_{2} \in H), \\
 &(v_{1}, v_{2})_{V}:=
 \int_{\Omega}\nabla v_{1}\cdot\nabla v_{2}\,dx + \int_{\Omega} v_{1}v_{2}\,dx \quad 
 (v_{1}, v_{2} \in V),
\end{align*}
 respectively,
 and with the related Hilbertian norms. 
 Moreover, we use the notation
   $$
   W:=\bigl\{z\in H^2(\Omega)\ |\ \partial_{\nu}z = 0 \quad 
   \mbox{a.e.\ on}\ \partial\Omega\bigr\}.
   $$ 
 The notation $V^{*}$ denotes the dual space of $V$ with 
 duality pairing $\langle\cdot, \cdot\rangle_{V^*, V}$. 
 Moreover, in this paper, the bijective mapping $F : V \to V^{*}$ and 
 the inner product in $V^{*}$ are defined as 
    \begin{align*}
    &\langle Fv_{1}, v_{2} \rangle_{V^*, V} := 
    (v_{1}, v_{2})_{V} \quad (v_{1}, v_{2}\in V),   
    %\label{defF}
    \\[1mm]
    &(v_{1}^{*}, v_{2}^{*})_{V^{*}} := 
    \left\langle v_{1}^{*}, F^{-1}v_{2}^{*} 
    \right\rangle_{V^*, V} 
    \quad (v_{1}^{*}, v_{2}^{*}\in V^{*}).   
    %\label{innerVstar}
    \end{align*}
This article employs the Hilbert space 
  $$
  V_{0}:=\left\{ z \in H^1(\Omega)\ \Big{|} \ \int_{\Omega} z = 0 \right\}   
  $$
 with inner product %(the same as in $V$)  
 \begin{align*} 
 (v_{1}, v_{2})_{V_{0}}:=
 \int_{\Omega}\nabla v_{1}\cdot\nabla v_{2}\,dx %+ \int_{\Omega} v_{1}v_{2}\,dx 
 \quad (v_{1}, v_{2} \in V_{0}) 
\end{align*}
 and with the related Hilbertian norm.
 The notation $V_{0}^{*}$ denotes the dual space of $V_{0}$ with 
 duality pairing $\langle\cdot, \cdot\rangle_{V_{0}^*, V_{0}}$. 
 Moreover, in this paper, the bijective mapping ${\cal N} : V_{0}^{*} \to V_{0}$ and 
 the inner product in $V_{0}^{*}$ are specified by  
    \begin{align*}
    &\langle v^{*}, v \rangle_{V_{0}^*, V_{0}} =:  
    \int_{\Omega} \nabla {\cal N}v^{*} \cdot \nabla v 
    \quad (v^{*} \in V_{0}^*, v \in V_{0}),   
    %\label{defN}
    \\[1mm]
    &(v_{1}^{*}, v_{2}^{*})_{V_{0}^{*}} := 
    \left\langle v_{1}^{*}, {\cal N}v_{2}^{*} 
    \right\rangle_{V_{0}^*, V_{0}} 
    \quad (v_{1}^{*}, v_{2}^{*}\in V_{0}^{*}).  
    %\label{innerVzerostar}
    \end{align*}

We define weak solutions of \ref{P} and \ref{Pep} as follows.
%-----------------------------------------------------------------
%
%         　　　　　Definition of weak solution of (P)
%
%-----------------------------------------------------------------
%
 \begin{df}         
 A pair $(u, \mu)$ with 
    \begin{align*}
    u\in H^1(0, T; V^{*})\cap L^{\infty}(0, T; L^4(\Omega)),\ \mu \in L^2(0, T; V)
    \end{align*}
 is called a {\it weak solution} of \ref{P} 
 if $(u, \mu)$ satisfies 
    \begin{align}
        &\langle u'(t), z \rangle_{V^{*}, V} + 
          \int_{\Omega} \nabla \mu(t) \cdot \nabla z 
          - \eta\int_{\Omega} 
                             u(t)\nabla(1-\Delta)^{-1}u(t) \cdot \nabla z  = 0 
         \notag \\[2mm] 
         &\hspace{81mm} \mbox{for all}\ z \in V\ 
          \mbox{and a.a}.\ t\in(0, T), 
          \label{de1}
     \\[4mm]
        & \mu = \beta(u) - f \quad \mbox{a.e.\ on}\ \Omega\times(0, T), \label{de2}
     \\[3mm]
        & u(0) = u_{0} \quad \mbox{a.e.\ on}\ \Omega. \label{de3}
     \end{align}
 \end{df}
%-----------------------------------------------------------------
%
%         　　Definition of weak solution of (P)_{\ep}
%
%-----------------------------------------------------------------
%
%
 \begin{df}       
 A pair $(u_{\ep}, \mu_{\ep})$ with 
    \begin{align*}
    &u_{\ep}\in H^1(0, T; V^{*})\cap L^{\infty}(0, T; V)\cap 
    L^2(0, T; W), 
    \\
    &\mu_{\ep}\in L^2(0, T; V)
    \end{align*}
 is called a {\it weak solution} of \ref{Pep} if 
 $(u_{\ep} ,\mu_{\ep})$ 
 satisfies 
    \begin{align}
        &\langle u_{\ep}'(t), z \rangle_{V^{*}, V} + 
          \int_{\Omega} \nabla \mu_{\ep}(t) \cdot \nabla z 
          - \eta\int_{\Omega} 
                             u_{\ep}(t)\nabla(1-\Delta)^{-1}u_{\ep}(t) \cdot \nabla z  = 0 
         \notag \\[2mm] 
         &\hspace{81mm} \mbox{for all}\ z \in V\ 
          \mbox{and a.a}.\ t\in(0, T), \label{de4}
     \\[4mm]
        & \mu_{\ep} = -\ep\Delta u_{\ep} + \beta(u_{\ep}) + \pi_{\ep}(u_{\ep}) - f  
        \quad \mbox{a.e.\ on}\ \Omega\times(0, T), \label{de5}
     \\[4mm]
        & u_{\ep}(0) = u_{0\ep} 
        \quad \mbox{a.e.\ on}\ \Omega. \label{de6}
     \end{align}
 \end{df}

Now the main results read as follows.
%==========================Main Theorem1==========================%
\begin{thm}\label{maintheorem1}
Assume {\rm (A1)-(A5)}. Then for all $\ep\in(0, 1]$                         
there exists a weak solution $(u_{\ep}, \mu_{\ep})$ of 
{\rm \ref{Pep}}.  
In addition, there exists a constant $M>0$, depending only on the data, 
such that 
     \begin{align}
     &\ep \|u_{\ep}(t)\|_{V}^2 + \|u_{\ep}(t)\|_{L^4(\Omega)}^4 \leq M, \label{epes1} 
     \\[2mm] 
     &\int_{0}^{t}\|u_{\ep}'(s)\|_{V^*}^2\,ds \leq M,  \label{epes2} 
     \\[2mm] 
     &\ep^2\int_{0}^{t}\|u_{\ep}(s)\|_{W}^2\,ds \leq M, \label{epes3} 
     \\[2mm] 
     &\int_{0}^{t}\|\mu_{\ep}(s)\|_{V}^2\,ds 
                    + \int_{0}^{t}\|\beta(u_{\ep}(s))\|_{H}^2\,ds \leq M \label{epes4}
     \end{align}
for all $t\in[0, T]$ and all $\ep\in(0, 1]$. 
\end{thm}
%
%
%==========================Main Theorem2==========================%
\begin{thm}\label{maintheorem2}
Assume {\rm (A1)-(A5)}. 
Then there exists a weak solution $(u, \mu)$ of {\rm \ref{P}}. 
\end{thm}
\begin{remark}
In the case that $\beta(r) = r$ and $d=1$ 
Osaki--Yagi \cite{Osaki-Yagi} 
established existence of a finite-dimensional attractor
and proved  
that global existence and boundedness hold for all smooth initial data,   
which implies that blow-up solutions do not exist in the 1-dimensional setting.  
In the case that $\beta$ is nonlinear 
Marinoschi \cite{M-2013} proved local existence of solutions to \ref{P} for unbounded initial data by applying the nonlinear semigroup theory. 
Moreover, Yokota--Yoshino \cite{YY2016} 
established not only local but also global existence of solutions to \ref{P} 
for unbounded initial data 
by improving the method used in \cite{M-2013},  
while this paper derives global existence of solutions to \ref{P} 
for unbounded initial data
by a Cahn--Hilliard approach. 
\end{remark}

This paper is organized as follows. 
Section \ref{Sec2} considers a suitable approximation of \ref{Pep}  
in terms of a parameter $\lambda > 0$ and introduces a time discretization scheme  
in reference to \cite{CK2}.   
In Section \ref{Sec3} 
we establish existence for the discrete problem. 
In Section \ref{Sec4} we derive uniform estimates for the time discrete solutions 
and show existence for the approximation of \ref{Pep} 
by passing to the limit as the time step tends to zero. 
Finally, in Section \ref{Sec5} 
we prove existence for \ref{Pep} and \ref{P}  
by taking the limit in the approximation of \ref{Pep} as $\lambda\searrow0$ 
and in \ref{Pep} as $\ep\searrow0$ individually.

\vspace{10pt}

%%==============================================================%%
%%==============                                  ==============%%
%%======                      Section2                    ======%%
%%====                                                      ====%%
%%==                                                          ==%%
%%====                                                      ====%%
%%======                                                  ======%%
%%==============                                  ==============%%
%%==============================================================%%
\section{Approximate problems and preliminaries}\label{Sec2}

We consider the approximation 
\begin{equation}\label{Peplam}\tag*{(P)$_{\ep, \lambda}$}
\begin{cases}
(u_{\ep, \lambda})_t - \Delta \mu_{\ep, \lambda} 
                   + \eta\nabla\cdot(u_{\ep, \lambda} \nabla v_{\ep, \lambda}) = 0 
&\mbox{in}\ \Omega \times (0, T), \\[2mm] 
\mu_{\ep, \lambda} = \lambda(u_{\ep, \lambda})_t - \ep\Delta u_{\ep, \lambda} 
                                  + \beta(u_{\ep, \lambda}) + \pi_{\ep}(u_{\ep, \lambda}) - f 
&\mbox{in}\ \Omega \times (0, T), \\[2mm]
0 = \Delta v_{\ep, \lambda} - v_{\ep, \lambda} + u_{\ep, \lambda} 
&\mbox{in}\ \Omega \times (0, T), \\[2mm]
\partial_{\nu} \mu_{\ep, \lambda} = \partial_{\nu} u_{\ep, \lambda} 
                                                          = \partial_{\nu} v_{\ep, \lambda} = 0 
&\mbox{on}\ \partial\Omega \times (0, T), \\[2mm] 
u_{\ep, \lambda} (0) = u_{0\ep} &\mbox{in}\ \Omega,  
\end{cases}
\end{equation}
where $\lambda \in (0, \ep)$. 
The definition of weak solutions to \ref{Peplam} is as follows. 
\begin{df}       
 A pair $(u_{\ep, \lambda}, \mu_{\ep, \lambda})$ with 
    \begin{align*}
    &u_{\ep, \lambda}\in H^1(0, T; H)\cap L^{\infty}(0, T; V)\cap 
    L^2(0, T; W), 
    \\
    &\mu_{\ep, \lambda}\in L^2(0, T; V)
    \end{align*}
 is called a {\it weak solution} of \ref{Peplam} if 
 $(u_{\ep, \lambda} ,\mu_{\ep, \lambda})$ 
 satisfies 
   \begin{align}
    &(u_{\ep, \lambda}'(t), z)_{H} 
    + \int_{\Omega} \nabla \mu_{\ep, \lambda}(t) \cdot \nabla z 
    - \eta\int_{\Omega} 
         u_{\ep, \lambda}(t)\nabla(1-\Delta)^{-1}u_{\ep, \lambda}(t) \cdot \nabla z  = 0 
         \notag \\[2mm] 
         &\hspace{81mm} \mbox{for all}\ z \in V\ 
          \mbox{and a.a}.\ t\in(0, T), \label{de7}
     \\[4mm]
        & \mu_{\ep, \lambda} 
           = \lambda u_{\ep, \lambda}' 
               -\ep\Delta u_{\ep, \lambda} + \beta(u_{\ep, \lambda}) 
               + \pi_{\ep}(u_{\ep, \lambda}) - f  
        \quad \mbox{a.e.\ on}\ \Omega\times(0, T), \label{de8}
     \\[4mm]
        & u_{\ep, \lambda}(0) = u_{0\ep} 
        \quad \mbox{a.e.\ on}\ \Omega. \label{de9}
   \end{align}
 \end{df}
\medskip

\noindent Moreover, to prove existence of weak solutions to \ref{Peplam} 
we employ a time discretization scheme. 
More precisely, 
in reference to \cite{CK2}, 
we will deal with the following problem: 
find $(u_{\lambda, n+1}, \mu_{\lambda, n+1}) \in W \times W$ such that  
\begin{equation*}\tag*{(P)$_{\lambda, n}$}\label{Plamn}
     \begin{cases}
        \delta_{h}u_{\lambda, n} 
         + h \delta_{h}\mu_{\lambda, n} - \Delta \mu_{\lambda, n+1} 
         + \eta\nabla\cdot(u_{\lambda, n}\nabla(1-\Delta)^{-1}u_{\lambda, n})  
         = 0  
         & \mbox{in}\ \Omega, 
 \\[2mm]
         \mu_{\lambda, n+1} = 
         \lambda\delta_{h}u_{\lambda, n} - \ep\Delta u_{\lambda, n+1} 
         + \beta (u_{\lambda, n+1}) + \pi_{\ep}(u_{\lambda, n+1}) - f_{n+1}    
         & \mbox{in}\ \Omega, 
 \\[2mm]
         \partial_{\nu}\mu_{\lambda, n+1} = \partial_{\nu}u_{\lambda, n+1} = 0                                   
         & \mbox{on}\ \partial\Omega 
     \end{cases}
\end{equation*}
for $n=0, ... , N-1$, where $h=\frac{T}{N}$, $N \in \mathbb{N}$, 
\begin{align*}%\label{delta1}
u_{\lambda, 0} := u_{0\ep}, \quad 
\mu_{\lambda, 0} := 0, \quad 
\delta_{h}u_{\lambda, n} := \frac{u_{\lambda, n+1}-u_{\lambda, n}}{h}, \quad 
\delta_{h}\mu_{\lambda, n} 
:= \frac{\mu_{\lambda, n+1} - \mu_{\lambda, n}}{h},  
\end{align*}
and $f_{k} := \frac{1}{h}\int_{(k-1)h}^{kh} f(s)\,ds$  
for $k = 1, ... , N$. 
Also, putting 
\begin{align}
&\hat{u}_{h}(0) := u_{\lambda, 0} = u_{0\ep},\ 
(\hat{u}_{h})_{t}(t) := \delta_{h}u_{\lambda, n},  \label{hat1} 
\\[2mm]  
&\hat{\mu}_{h}(0) := \mu_{\lambda, 0} = 0,\ 
(\hat{\mu}_{h})_{t}(t) := \delta_{h}\mu_{\lambda, n},  \label{hat2}
\\[2mm] 
&\overline{u}_{h} (t) := u_{\lambda, n+1},\ 
\underline{u}_{h} (t) := u_{\lambda, n},\ 
\overline{\mu}_{h} (t) := \mu_{\lambda, n+1},\ 
\overline{f}_{h}(t) := f_{n+1}  
\label{line1}   
\end{align}
for a.a.\ $t \in (nh, (n+1)h)$, $n=0, ..., N-1$, 
we can rewrite \ref{Plamn} as  
\begin{equation*}\tag*{(P)$_{h}$}\label{Ph}
     \begin{cases}
        (\hat{u}_{h})_{t} + h(\hat{\mu}_{h})_{t} - \Delta\overline{\mu}_{h} 
        + \eta\nabla\cdot(\underline{u}_{h}\nabla(1-\Delta)^{-1}\underline{u}_{h}) 
        = 0 
         & \mbox{in}\ \Omega\times(0, T), 
 \\[2mm]
         \overline{\mu}_{h} = \lambda(\hat{u}_{h})_{t}  
         - \ep\Delta\overline{u}_{h} 
         + \beta(\overline{u}_{h}) 
         + \pi_{\ep}(\overline{u}_{h}) 
         - \overline{f}_{h} 
         & \mbox{in}\ \Omega\times(0, T), 
 \\[3mm]
         \partial_{\nu}\overline{\mu}_{h} 
         = \partial_{\nu}\overline{u}_{h} = 0                                   
         & \mbox{on}\ \partial\Omega\times(0, T),
 \\[2mm]
        \hat{u}_{h}(0) = u_{0\ep},\  \hat{\mu}_{h}(0) = 0                                    
         & \mbox{in}\ \Omega.  
     \end{cases}
 \end{equation*}

\begin{remark}
On account of \eqref{hat1}-\eqref{line1}, 
the reader can check directly the following properties: 
\begin{align}
&\|\hat{u}_{h}\|_{L^{\infty}(0, T; L^4(\Omega))} 
= \max\{\|u_{0\ep}\|_{L^4(\Omega)},\ 
                                       \|\overline{u}_{h}\|_{L^{\infty}(0, T; L^4(\Omega))}\}, 
\label{rem1} 
\\[2mm] 
&\|\hat{u}_{h}\|_{L^{\infty}(0, T; V)} 
= \max\{\|u_{0\ep}\|_{V},\ \|\overline{u}_{h}\|_{L^{\infty}(0, T; V)}\}, 
\label{rem2} 
\\[2mm] 
&\|\hat{\mu}_{h}\|_{L^{\infty}(0, T; H)} = \|\overline{\mu}_{h}\|_{L^{\infty}(0, T; H)}, 
\label{rem3} 
\\[2mm] 
&\|\overline{u}_{h} - \hat{u}_{h}\|_{L^2(0, T; H)}^2 
   = \frac{h^2}{3}\|(\hat{u}_{h})_{t}\|_{L^2(0, T; H)}^2, 
\label{rem4}
\\[2mm]
&h(\hat{u}_{h})_{t} = \overline{u}_{h} - \underline{u}_{h}. \label{rem5}
\end{align}
\end{remark} 
\begin{lem}\label{existenceforPlamn}
For all $\ep \in (0, 1]$, $\lambda \in (0, \ep)$, $h \in (0, \frac{\lambda}{2c_{3}\ep})$ 
there exists a unique solution $(u_{\lambda, n+1}, \mu_{\lambda, n+1})$ 
of {\rm \ref{Plamn}} satisfying 
\[
u_{\lambda, n+1}, \mu_{\lambda, n+1} \in W \quad 
\mbox{for}\ n = 0, ..., N-1. 
\]
\end{lem}
\begin{lem}\label{existenceforPeplam}
For all $\ep \in (0, 1]$ and all $\lambda \in (0, \ep)$ 
there exists a weak solution of {\rm \ref{Peplam}}. 
\end{lem}
\medskip

We provide some basic results which will be applied in this paper. 
\begin{lem}\label{PT}
 Let $\beta : \mathbb{R} \to \mathbb{R}$ 
 be a multi-valued maximal monotone function.     
 Then
    \begin{align*}
    \bigl(-\Delta u, \beta_{\tau}(u)\bigr)_{H} \geq 0 
    \quad \mbox{for all}\ 
    u\in W\ \mbox{and all}\ \tau > 0,  
    \end{align*}
 where 
 $\beta_{\tau}$ is the Yosida approximation of $\beta$ on $\mathbb{R}$. 
 In particular, if $\beta : \mathbb{R} \to \mathbb{R}$ 
 is a single-valued maximal monotone function, 
 then 
 \begin{align*}
    \bigl(-\Delta u, \beta(u)\bigr)_{H} \geq 0 
    \quad \mbox{for all}\ 
    u\in W\ \mbox{with}\ \beta(u)\in H. 
    \end{align*}
 \end{lem}
 \begin{proof}
 It follows from 
 Okazawa \cite[Proof of Theorem 3 with $a=b=0$]{O-1983}   
 that 
 \[
 \bigl(-\Delta u, \beta_{\tau}(u)\bigr)_{H} \geq 0 
    \quad \mbox{for all}\ 
    u\in W\ \mbox{and all}\ \tau>0.
 \]
 In the case that $\beta : \mathbb{R} \to \mathbb{R}$ 
 is a single-valued maximal monotone function,  
 noting that $\beta_{\tau}(u) \to \beta(u)$ in $H$  
 as $\tau \searrow 0$ if $\beta(u) \in H$ 
 (see e.g., {\cite[Proposition 2.6]{Brezis}} 
 or \cite[Theorem IV.1.1]{S-1997}),   
 we can obtain the second inequality.     
 \end{proof}
\begin{lem}[{\cite[Section 8, Corollary 4]{Si-1987}}] \label{ALlem}                                                                   
 Assume that 
 $$
 X \hookrightarrow Z \hookrightarrow Y \ 
 \mbox{with compact embedding}\ X \hookrightarrow Z\ 
 \mbox{$($$X$, $Z$ and $Y$ are Banach spaces$)$.}
 $$
 \begin{enumerate}
 \setlength{\itemsep}{-0.5mm}
 \item[{\rm (i)}] Let $K$ be bounded in $L^p(0, T; X)$ and 
 let $\{\frac{dv}{dt}\ |\ v\in K \}$  
 be bounded in $L^1(0, T; Y)$ with some constant $1\leq p<\infty$. 
 Then $K$ is relatively compact in $L^p(0, T; Z)$. 
 \item[{\rm (ii)}] Let $K$ be bounded in $L^{\infty}(0, T; X)$ and 
 let $\{\frac{dv}{dt}\ |\ v\in K \}$   
 be bounded in $L^r(0, T; Y)$ with some constant $r>1$. 
 Then $K$ is relatively compact in $C([0, T]; Z)$. 
 \end{enumerate}
 \end{lem}

\vspace{10pt}

%%==============================================================%%
%%==============                                  ==============%%
%%======                      Section3                    ======%%
%%====                                                      ====%%
%%==                                                          ==%%
%%====                                                       ====%%
%%======                                                  ======%%
%%==============                                  ==============%%
%%==============================================================%%
\section{Existence for the discrete problem}\label{Sec3}
In this section we will prove Lemma \ref{existenceforPlamn}. 
\begin{lem}\label{elliptic1}
For all $g \in H$, $\ep \in (0, 1]$, $\lambda \in (0, \ep)$, 
$h \in (0, \frac{\lambda}{2c_{3}\ep})$
there exists a unique solution $u \in W$ of the equation 
\[
(\lambda + (1-\Delta)^{-1})u - \ep h \Delta u + h\beta(u) + h\pi_{\ep}(u) = g. 
\]
\end{lem}
\begin{proof}
We define the operator $\Psi : V \to V^{*}$ as 
\[
\langle \Psi u, z \rangle_{V^{*}, V} 
:= ((\lambda + (1-\Delta)^{-1})u, z)_{H} 
   + \ep h \int_{\Omega}\nabla u \cdot \nabla z 
   + h(\beta_{\tau}(u), z)_{H} + h(\pi_{\ep}(u), z)_{H} 
\]
for $u, z \in V$, 
where $\tau>0$ and $\beta_{\tau}$ is the Yosida approximation of $\beta$ 
on $\mathbb{R}$. 
Then this operator is monotone, continuous and coercive 
for all $\ep \in (0, 1]$, $\lambda \in (0, \ep)$, $h \in (0, \frac{\lambda}{c_{3}\ep})$, 
$\tau>0$. 
Indeed, we have from (A4), 
the monotonicity of $(1-\Delta)^{-1}$ and $\beta_{\tau}$,  
the Lipschitz continuity of $(1-\Delta)^{-1}$ and $\beta_{\tau}$ 
that 
\begin{align*}
&\langle \Psi u - \Psi w, u -w \rangle_{V^{*}, V} 
\\
&= \lambda\|u-w\|_{H}^2 + ((1-\Delta)^{-1}(u-w), u-w)_{H} 
 + \ep h\|\nabla(u-w)\|_{H}^2   
\\
&\,\quad+ h (\beta_{\tau}(u)-\beta_{\tau}(w), u-w)_{H} 
 + h (\pi_{\ep}(u) -\pi_{\ep}(w), u-w)_{H} 
\\ 
&\geq \min\{\lambda-c_{3}\ep h, \ep h\}\|u-w\|_{V}^2, 
\\[4mm]
&\langle \Psi u - \Psi w, z \rangle_{V^{*}, V} 
\\ 
&= \lambda(u-w, z)_{H} + ((1-\Delta)^{-1}(u-w), z)_{H} 
 + \ep h\int_{\Omega}\nabla (u-w) \cdot \nabla z   
\\
&\,\quad+ h (\beta_{\tau}(u)-\beta_{\tau}(w), z)_{H} 
 + h (\pi_{\ep}(u) -\pi_{\ep}(w), z)_{H} 
\\ 
&\leq \max\{\lambda, 1, \ep h, \tau^{-1}h, c_{3}\ep h\}\|u-w\|_{V}\|z\|_{V}
\end{align*}
and 
\begin{align*}
\langle \Psi u - \Psi 0, u \rangle_{V^{*}, V}  
\geq \min\{\lambda-c_{3}\ep h, \ep h\}\|u\|_{V}^2
\end{align*}
for all $u, w, z \in V$, $\ep \in (0, 1]$, $\lambda \in (0, \ep)$, 
$h \in (0, \frac{\lambda}{c_{3}\ep})$, $\tau>0$. 
Thus the operator $\Psi : V \to V^{*}$
is surjective
for all $\ep \in (0, 1]$, $\lambda \in (0, \ep)$, 
$h \in (0, \frac{\lambda}{c_{3}\ep})$, $\tau>0$ 
(see e.g., \cite[p.\ 37]{Barbu2})  
and then the elliptic regularity theory yields that
for all $g \in H$, $\ep \in (0, 1]$, $\lambda \in (0, \ep)$, 
$h \in (0, \frac{\lambda}{c_{3}\ep})$, $\tau>0$    
there exists a unique solution $u_{\tau} \in W$ of the equation
\begin{align}\label{ep elliptic eq}
(\lambda + (1-\Delta)^{-1})u_{\tau} - \ep h \Delta u_{\tau} 
+ h\beta_{\tau}(u_{\tau}) + h\pi_{\ep}(u_{\tau}) = g. 
\end{align}
Here, multiplying \eqref{ep elliptic eq} by $u_{\tau}$
and integrating over $\Omega$,
we see from (A4) and the Young inequality that 
\begin{align*}
&\lambda\|u_{\tau}\|_{H}^2 
+ ((1-\Delta)^{-1}u_{\tau}, u_{\tau})_{H} 
+ \ep h \|\nabla u_{\tau}\|_{H}^2 
+ h(\beta_{\tau}(u_{\tau}), u_{\tau})_{H} 
\\ 
&= (g, u_{\tau})_{H} 
   - h(\pi_{\ep}(u_{\tau})-\pi_{\ep}(0), u_{\tau})_{H}
  - h(\pi_{\ep}(0), u_{\tau})_{H} 
\\ 
&\leq \|g\|_{H}\|u_{\tau}\|_{H} 
         + h(\|\pi_{\ep}'\|_{L^{\infty}(\mathbb{R})} + |\pi_{\ep}(0)|)\|u_{\tau}\|_{H}^2 
         + \frac{|\pi_{\ep}(0)||\Omega|}{4}h  
\\ 
&\leq \|g\|_{H}\|u_{\tau}\|_{H} + c_{3}\ep h\|u_{\tau}\|_{H}^2 
         + \frac{c_{3}\ep|\Omega|}{4}h  
\\ 
&\leq \frac{1}{4c_{3}\ep h}\|g\|_{H}^2 + 2c_{3}\ep h\|u_{\tau}\|_{H}^2 
         + \frac{c_{3}\ep|\Omega|}{4}h,   
\end{align*} 
and hence 
for all $\ep \in (0, 1]$, $\lambda \in (0, \ep)$, $h \in (0, \frac{\lambda}{2c_{3} \ep})$ 
there exists a constant $C_{1} = C_{1}(\ep, \lambda, h)$ such that 
\begin{equation}\label{ellipes1}
\|u_{\tau}\|_{V}^2 \leq C_{1} 
\end{equation}
for all $\tau > 0$. 
It follows from \eqref{ep elliptic eq}, Lemma \ref{PT}, the monotonicity of $\beta_{\tau}$ 
and the Lipschitz continuity of the operator $(1-\Delta)^{-1}$  
that 
\begin{align}\label{a1}
&\|\beta_{\tau}(u_{\tau})\|_{H}^2 
= (\beta_{\tau}(u_{\tau}), \beta_{\tau}(u_{\tau}))_{H} 
\notag \\ 
&= \frac{1}{h}(g, \beta_{\tau}(u_{\tau}))_{H} 
   - (\pi_{\ep}(u_{\tau}), \beta_{\tau}(u_{\tau}))_{H} 
   - \ep(-\Delta u_{\tau}, \beta_{\tau}(u_{\tau}))_{H} 
\notag \\
&\,\quad- \frac{\lambda}{h}(u_{\tau}, \beta_{\tau}(u_{\tau}))_{H} 
   -  \frac{1}{h}((1-\Delta)^{-1}u_{\tau}, \beta_{\tau}(u_{\tau}))_{H} 
\notag \\ 
&\leq \frac{1}{h}\|g\|_{H}\|\beta_{\tau}(u_{\tau})\|_{H} 
         + \|\pi_{\ep}(u_{\tau})\|_{H}\|\beta_{\tau}(u_{\tau})\|_{H} 
         + \frac{1}{h}\|u_{\tau}\|_{H}\|\beta_{\tau}(u_{\tau})\|_{H}.  
\end{align}
Thus we derive from \eqref{a1}, (A4), the Young inequality and \eqref{ellipes1} that 
for all $\ep \in (0, 1]$, $\lambda \in (0, \ep)$, $h \in (0, \frac{\lambda}{2c_{3} \ep})$ 
there exists a constant $C_{2} = C_{2}(\ep, \lambda, h)$ such that 
\begin{align}\label{ellipes2}
\|\beta_{\tau}(u_{\tau})\|_{H}^2 \leq C_{2} 
\end{align}
for all $\tau > 0$. 
Moreover, the equation \eqref{ep elliptic eq}, the condition (A4), 
the inequalities \eqref{ellipes1}, \eqref{ellipes2},  
and the elliptic regularity theory  
imply that 
for all $\ep \in (0, 1]$, $\lambda \in (0, \ep)$, $h \in (0, \frac{\lambda}{2c_{3} \ep})$ 
there exists a constant $C_{3} = C_{3}(\ep, \lambda, h)$ such that  
\begin{align}\label{ellipes3}
\|u_{\tau}\|_{W} ^2 \leq C_{3}
\end{align}
for all $\tau>0$.  
Hence by \eqref{ellipes2} and \eqref{ellipes3} there exist 
$u \in W$ and $\xi \in H$ such that 
\begin{align}
&u_{\tau} \to u \quad \mbox{weakly in}\ W,  \label{weakconvforellip1} \\ 
&\beta_{\tau}(u_{\tau}) \to \xi \quad \mbox{weakly in}\ H  
\label{weakconvforellip2}
\end{align}
as $\tau = \tau_{j} \searrow 0$. 
Here, owing to \eqref{ellipes1} and the compact embedding $V \hookrightarrow H$, 
it holds that 
\begin{align}\label{strongconvforellip}
u_{\tau} \to u \quad \mbox{strongly in}\ H 
\end{align}
as $\tau = \tau_{j} \searrow 0$. 
Moreover, the convergences \eqref{weakconvforellip2} and 
\eqref{strongconvforellip} 
mean that 
$(\beta_{\tau}(u_{\tau}), u_{\tau})_{H} \to (\xi, u)_{H}$ 
as $\tau = \tau_{j} \searrow 0$, whence 
we have that 
\begin{align}\label{xibetaforellip}
\xi = \beta(u)  \quad \mbox{a.e.\ on}\ \Omega
\end{align}
(see e.g., \cite[Lemma 1.3, p.\ 42]{Barbu1}).   

Therefore we infer from \eqref{ep elliptic eq}, \eqref{weakconvforellip1}, 
\eqref{weakconvforellip2}, \eqref{strongconvforellip}, 
the Lipschitz continuity of $\pi_{\ep}$, and \eqref{xibetaforellip}  
that there exists a solution $u \in W$ of the equation 
\begin{align*}
(\lambda + (1-\Delta)^{-1})u - \ep h \Delta u + h\beta(u) + h\pi_{\ep}(u) = g. 
\end{align*}
Moreover, we can check that the solution $u$ of this problem is unique. 
\end{proof}
\begin{prlem2.1}
The problem \ref{Plamn} can be written as 
\begin{equation}\label{Qlamn}\tag*{(Q)$_{\lambda, n}$}
\begin{cases}
(\lambda + (1-\Delta)^{-1})u_{\lambda, n+1} 
-\ep h \Delta u_{\lambda, n+1} 
+ h\beta(u_{\lambda, n+1}) + h\pi_{\ep}(u_{\lambda, n+1}) 
\\ 
\hspace{0.7mm}= hf_{n+1} + \lambda u_{\lambda, n} + (1-\Delta)^{-1}u_{\lambda, n} 
   + h(1-\Delta)^{-1}\mu_{\lambda, n}  
\\ 
\hspace{40mm}- \eta h(1-\Delta)^{-1}
                      \nabla\cdot(u_{\lambda, n}\nabla(1-\Delta)^{-1}u_{\lambda, n}), & 
\\[5mm]
\mu_{\lambda, n+1} - \Delta \mu_{\lambda, n+1} 
= \mu_{\lambda, n} - \frac{u_{\lambda, n+1}-u_{\lambda, n}}{h} 
   - \eta\nabla\cdot(u_{\lambda, n}\nabla(1-\Delta)^{-1}u_{\lambda, n}). &
\end{cases}
\end{equation}
Thus proving Lemma \ref{existenceforPlamn} is equivalent to 
derive existence and uniqueness of solutions to \ref{Qlamn} 
for $n = 0, ..., N-1$. 
It suffices to consider the case that $n=0$. 
By Lemma \ref{elliptic1} there exists a unique solution $u_{\lambda, 1} \in W$ 
of the equation 
\begin{align*}
&(\lambda + (1-\Delta)^{-1})u_{\lambda, 1} 
-\ep h \Delta u_{\lambda, 1} 
+ h\beta(u_{\lambda, 1}) + h\pi_{\ep}(u_{\lambda, 1}) 
\\[2mm] 
&= hf_{1} + \lambda u_{\lambda, 0} + (1-\Delta)^{-1}u_{\lambda, 0} 
   + h(1-\Delta)^{-1}\mu_{\lambda, 0}  
\\ 
&\,\quad- \eta h(1-\Delta)^{-1}
                      \nabla\cdot(u_{\lambda, 0}\nabla(1-\Delta)^{-1}u_{\lambda, 0}). 
\end{align*}
Therefore, putting $\mu_{1} := 
(1-\Delta)^{-1}(\mu_{\lambda, 0} - \frac{u_{\lambda, 1}-u_{\lambda, 0}}{h} 
   - \eta\nabla\cdot(u_{\lambda, 0}\nabla(1-\Delta)^{-1}u_{\lambda, 0}))$, 
we conclude that 
there exists a unique solution $(u_{\lambda, 1}, \mu_{\lambda, 1})$ 
of \ref{Qlamn} in the case that $n=0$. \qed
\end{prlem2.1}

\vspace{10pt}
 
%%==============================================================%%
%%==============                                  ==============%%
%%======                      Section4                    ======%%
%%====                                                      ====%%
%%==                                                          ==%%
%%==== Uniform estimates and passage to the limit  ====%%
%%======                                                  ======%%
%%==============                                  ==============%%
%%==============================================================%%
\section{Uniform estimates and passage to the limit}\label{Sec4}

In this section we will show Lemma \ref{existenceforPeplam}. 
We will establish a priori estimates for \ref{Ph} to 
prove existence for \ref{Peplam} 
by passing to the limit in \ref{Ph}. 
\begin{lem}\label{timedisces1}
There exists a constant $C>0$ depending on the data 
such that 
for all $\ep \in (0, 1]$ and all $\lambda \in (0, \ep)$ 
there exists $h_{1} \in (0, \min\{1, \frac{\lambda}{2c_{3}\ep}\})$ such that 
\begin{align*} 
&\|(\hat{u}_{h})_{t} + h(\hat{\mu}_{h})_{t}\|_{L^2(0, T; V_{0}^{*})}^2 
+ \lambda\|(\hat{u}_{h})_{t}\|_{L^2(0, T; H)}^2 
\\[2mm] 
&+ \ep\|\overline{u}_{h}\|_{L^{\infty}(0, T; V)}^2 
+ \ep h\|(\hat{u}_{h})_{t}\|_{L^2(0, T; V)}^2 
+ \|\overline{u}_{h}\|_{L^{\infty}(0, T; L^4(\Omega))}^4 
\\[1.5mm] 
&+ h\|\overline{\mu}_{h}\|_{L^{\infty}(0, T; H)}^2 
+ h^2\|(\hat{\mu}_{h})_{t}\|_{L^2(0, T; H)}^2  
+ \|\nabla\overline{\mu}_{h}\|_{L^2(0, T; H)}^2 
\leq C 
\end{align*}
for all $h \in (0, h_{1})$.  
\end{lem}
\begin{proof}
We multiply the first equation in \ref{Plamn} by $h\mu_{\lambda, n+1}$, 
integrate over $\Omega$ and use the Young inequality to obtain that 
\begin{align}\label{b1}
&(u_{\lambda, n+1}-u_{\lambda, n}, \mu_{\lambda, n+1})_{H} 
+ \frac{h}{2}(\|\mu_{\lambda, n+1}\|_{H}^2 - \|\mu_{\lambda, n}\|_{H}^2 
                                + \|\mu_{\lambda, n+1} - \mu_{\lambda, n}\|_{H}^2) 
\notag \\ 
&+ h\|\nabla\mu_{\lambda, n+1}\|_{H}^2 
\notag \\ 
&= \eta h\int_{\Omega}
      u_{\lambda, n}\nabla(1-\Delta)^{-1}u_{\lambda, n}\cdot\nabla\mu_{\lambda, n+1} 
\notag \\ 
&\leq \frac{h}{4}\|\nabla\mu_{\lambda, n+1}\|_{H}^2 
       + \eta^2 h\int_{\Omega}|u_{\lambda, n}\nabla(1-\Delta)^{-1}u_{\lambda, n}|^2.  
\end{align}
Here multiplying the second equation in \ref{Plamn} 
by $u_{\lambda, n+1}-u_{\lambda, n}$ 
and integrating over $\Omega$ 
lead to the identity 
\begin{align}\label{b2}
&(u_{\lambda, n+1}-u_{\lambda, n}, \mu_{\lambda, n+1})_{H} 
\notag \\ 
&= \lambda h \Bigl\|\frac{u_{\lambda, n+1}-u_{\lambda, n}}{h}\Bigr\|_{H}^2 
     + \frac{\ep}{2}(\|u_{\lambda, n+1}\|_{V}^2 - \|u_{\lambda, n}\|_{V}^2
                                                      + \|u_{\lambda, n+1}-u_{\lambda, n}\|_{V}^2) 
\notag \\
&\,\quad + (\beta(u_{\lambda, n+1}), u_{\lambda, n+1}-u_{\lambda, n})_{H} 
\notag \\
&\,\quad+ (\pi_{\ep}(u_{\lambda, n+1})-f_{n+1}-\ep u_{\lambda, n+1}, 
                                                               u_{\lambda, n+1}-u_{\lambda, n})_{H}. 
\end{align}
It follows from (A1) and the definition of the subdifferential that 
\begin{align}\label{b3} 
(\beta(u_{\lambda, n+1}), u_{\lambda, n+1}-u_{\lambda, n})_{H} 
\geq \int_{\Omega}\hat{\beta}(u_{\lambda, n+1}) 
       - \int_{\Omega}\hat{\beta}(u_{\lambda, n}). 
\end{align}
Since the first equation in \ref{Plamn} means that 
\begin{align}\label{Plamnhozon}
\int_{\Omega}\Bigl(\frac{u_{\lambda, n+1}-u_{\lambda, n}}{h} 
                                    + \mu_{\lambda, n+1}-\mu_{\lambda, n}\Bigr) = 0,   
\end{align}
we have from (A4) and the Young inequality that 
there exists a constant $C_{1}>0$ 
such that 
\begin{align}\label{b4}
&- (\pi_{\ep}(u_{\lambda, n+1})-f_{n+1}-\ep u_{\lambda, n+1}, 
                                                          u_{\lambda, n+1}-u_{\lambda, n})_{H} 
\notag \\[3mm] 
&= - h\Bigl(\pi_{\ep}(u_{\lambda, n+1})-f_{n+1}-\ep u_{\lambda, n+1}, 
                                    \frac{u_{\lambda, n+1}-u_{\lambda, n}}{h} 
                                    + \mu_{\lambda, n+1}-\mu_{\lambda, n}\Bigr)_{H} 
\notag \\ 
&\,\quad + h(\pi_{\ep}(u_{\lambda, n+1})-f_{n+1}-\ep u_{\lambda, n+1}, 
                                             \mu_{\lambda, n+1}-\mu_{\lambda, n})_{H} 
\notag \\[3mm] 
&= - h \Bigl\langle 
             \frac{u_{\lambda, n+1}-u_{\lambda, n}}{h} 
             + \mu_{\lambda, n+1}-\mu_{\lambda, n}, 
\notag \\
    &\hspace{20mm}\pi_{\ep}(u_{\lambda, n+1})-f_{n+1}-\ep u_{\lambda, n+1} 
                  - \frac{1}{|\Omega|}\int_{\Omega}
                                     (\pi_{\ep}(u_{\lambda, n+1})-f_{n+1}-\ep u_{\lambda, n+1})          
          \Bigr\rangle_{V_{0}^{*}, V_{0}} 
\notag \\ 
&\,\quad + h(\pi_{\ep}(u_{\lambda, n+1})-f_{n+1}-\ep u_{\lambda, n+1}, 
                                             \mu_{\lambda, n+1}-\mu_{\lambda, n})_{H} 
\notag \\[5mm] 
&\leq \frac{h}{8}\Bigl\|\frac{u_{\lambda, n+1}-u_{\lambda, n}}{h} 
             + \mu_{\lambda, n+1}-\mu_{\lambda, n}\Bigr\|_{V_{0}^{*}}^2 
         + \frac{h}{4}\|\mu_{\lambda, n+1}-\mu_{\lambda, n}\|_{H}^2 
         + C_{1}\ep^2 h \|u_{\lambda, n+1}\|_{V}^2 
\notag \\
         &\,\quad + C_{1}h\|f_{n+1}\|_{V}^2 + C_{1}\ep^2 h 
\end{align}
for all 
$\ep \in (0, 1]$, $\lambda \in (0, \ep)$, $h \in (0, \frac{\lambda}{2c_{3}\ep})$. 
We see from the first equation in \ref{Plamn} and the Young inequality that 
\begin{align}\label{b5}
&h\Bigl\|\frac{u_{\lambda, n+1}-u_{\lambda, n}}{h} 
             + \mu_{\lambda, n+1}-\mu_{\lambda, n}\Bigr\|_{V_{0}^{*}}^2 
\notag \\ 
&= h\Bigl\langle
            \frac{u_{\lambda, n+1}-u_{\lambda, n}}{h} 
                   + \mu_{\lambda, n+1}-\mu_{\lambda, n}, 
            {\cal N}\Bigl(\frac{u_{\lambda, n+1}-u_{\lambda, n}}{h} 
                                    + \mu_{\lambda, n+1}-\mu_{\lambda, n}\Bigr) 
      \Bigr\rangle_{V_{0}^{*}, V_{0}} 
\notag \\ 
&= -h\int_{\Omega}\nabla\mu_{\lambda, n+1}\cdot
                      \nabla{\cal N}\Bigl(\frac{u_{\lambda, n+1}-u_{\lambda, n}}{h} 
                                                 + \mu_{\lambda, n+1}-\mu_{\lambda, n}\Bigr) 
\notag \\ 
   &\,\quad + \eta h\int_{\Omega}
                            u_{\lambda, n}\nabla(1-\Delta)^{-1}u_{\lambda, n}\cdot
                               \nabla{\cal N}\Bigl(\frac{u_{\lambda, n+1}-u_{\lambda, n}}{h} 
                                                     + \mu_{\lambda, n+1}-\mu_{\lambda, n}\Bigr) 
\notag \\ 
&\leq \frac{h}{2}\|\nabla\mu_{\lambda, n+1}\|_{H}^2 
         + \eta^2 h\int_{\Omega}|u_{\lambda, n}\nabla(1-\Delta)^{-1}u_{\lambda, n}|^2 
\notag \\ 
  &\,\quad+ \frac{3}{4}h\Bigl\|\frac{u_{\lambda, n+1}-u_{\lambda, n}}{h} 
                                   + \mu_{\lambda, n+1}-\mu_{\lambda, n}\Bigr\|_{V_{0}^{*}}^2. 
\end{align}
We infer from 
the continuity of 
the embedding $W^{2, 4}(\Omega) \hookrightarrow W^{1, 4}(\Omega)$ 
and standard elliptic regularity theory (\cite{F-1969})    
that 
there exist constants $C_{2}, C_{3}>0$ such that 
\begin{align}\label{b6}
\eta^2 h\int_{\Omega}|u_{\lambda, n}\nabla(1-\Delta)^{-1}u_{\lambda, n}|^2 
&\leq \eta^2 h \|u_{\lambda, n}\|_{L^4(\Omega)}^2
                     \|\nabla(1-\Delta)^{-1}u_{\lambda, n}\|_{L^4(\Omega)}^2 
\notag \\ 
&\leq C_{2}\eta^2 h \|u_{\lambda, n}\|_{L^4(\Omega)}^2
                     \|(1-\Delta)^{-1}u_{\lambda, n}\|_{W^{2, 4}(\Omega)}^2 
\notag \\ 
&\leq C_{3}\eta^2 h \|u_{\lambda, n}\|_{L^4(\Omega)}^2
                     \|(1-\Delta)(1-\Delta)^{-1}u_{\lambda, n}\|_{L^4(\Omega)}^2 
\notag \\ 
&= C_{3}\eta^2 h \|u_{\lambda, n}\|_{L^4(\Omega)}^4
\end{align}
for all 
$\ep \in (0, 1]$, $\lambda \in (0, \ep)$, $h \in (0, \frac{\lambda}{2c_{3}\ep})$. 
Thus we derive from \eqref{b1}-\eqref{b3}, \eqref{b4}-\eqref{b6} that 
\begin{align}\label{b7}
&\frac{h}{8}\Bigl\|\frac{u_{\lambda, n+1}-u_{\lambda, n}}{h} 
                                   + \mu_{\lambda, n+1}-\mu_{\lambda, n}\Bigr\|_{V_{0}^{*}}^2 
+ \lambda h\Bigl\|\frac{u_{\lambda, n+1}-u_{\lambda, n}}{h}\Bigr\|_{H}^2 
\notag \\ 
&+ \frac{\ep}{2}(\|u_{\lambda, n+1}\|_{V}^2 - \|u_{\lambda, n}\|_{V}^2 
                                      + \|u_{\lambda, n+1}-u_{\lambda, n}\|_{V}^2) 
+ \int_{\Omega}\hat{\beta}(u_{\lambda, n+1}) 
                                       - \int_{\Omega}\hat{\beta}(u_{\lambda, n}) 
\notag \\ 
&+ \frac{h}{2}(\|\mu_{\lambda, n+1}\|_{H}^2 - \|\mu_{\lambda, n}\|_{H}^2) 
   + \frac{h}{4}\|\mu_{\lambda, n+1}-\mu_{\lambda, n}\|_{H}^2 
   + \frac{h}{4}\|\nabla\mu_{\lambda, n+1}\|_{H}^2  
\notag \\ 
&\leq C_{1}\ep h\|u_{\lambda, n+1}\|_{V}^2 
        + 2C_{3}\eta^2 h\|u_{\lambda, n}\|_{L^4(\Omega)}^4 
        + C_{1}h\|f_{n+1}\|_{V}^2 + C_{1}h 
\end{align}
for all 
$\ep \in (0, 1]$, $\lambda \in (0, \ep)$, $h \in (0, \frac{\lambda}{2c_{3}\ep})$. 
Then we sum \eqref{b7} over $n=0, ..., m-1$ with $1 \leq m \leq N$ 
and use (A5) 
to obtain that 
\begin{align*}
&\frac{h}{8}\sum_{n=0}^{m-1}
              \Bigl\|\frac{u_{\lambda, n+1}-u_{\lambda, n}}{h} 
                                   + \mu_{\lambda, n+1}-\mu_{\lambda, n}\Bigr\|_{V_{0}^{*}}^2 
+ \lambda h\sum_{n=0}^{m-1}
                     \Bigl\|\frac{u_{\lambda, n+1}-u_{\lambda, n}}{h}\Bigr\|_{H}^2 
\notag \\ 
&+ \ep\Bigl(\frac{1}{2} - C_{1} h \Bigr)\|u_{\lambda, m}\|_{V}^2 
   + \frac{\ep}{2}h^2 \sum_{n=0}^{m-1} 
                        \Bigl\|\frac{u_{\lambda, n+1}-u_{\lambda, n}}{h}\Bigr\|_{V}^2
   + \int_{\Omega}\hat{\beta}(u_{\lambda, m}) 
\notag \\ 
&+ \frac{h}{2}\|\mu_{\lambda, m}\|_{H}^2  
   + \frac{h^3}{4}\sum_{n=0}^{m-1}
            \Bigl\|\frac{\mu_{\lambda, n+1}-\mu_{\lambda, n}}{h}\Bigr\|_{H}^2 
   + \frac{h}{4}\sum_{n=0}^{m-1}\|\nabla\mu_{\lambda, n+1}\|_{H}^2  
\notag \\ 
&\leq C_{1}\ep h\sum_{j=0}^{m-1}\|u_{\lambda, j}\|_{V}^2 
        + 2C_{3}\eta^2 h\sum_{j=0}^{m-1}\|u_{\lambda, j}\|_{L^4(\Omega)}^4 
        + C_{1}h\sum_{n=0}^{m-1}\|f_{n+1}\|_{V}^2 + C_{1}T 
        + 2c_{4} 
\end{align*} 
for all 
$\ep \in (0, 1]$, $\lambda \in (0, \ep)$, $h \in (0, \frac{\lambda}{2c_{3}\ep})$ 
and $m=1, ..., N$,  
whence by (A2)
there exists a constant $C_{4}>0$ such that 
for all $\ep \in (0, 1]$ and all $\lambda \in (0, \ep)$ 
there exists $h_{1} \in (0, \min\{1, \frac{\lambda}{2c_{3}\ep}\})$ such that 
\begin{align*}
&h\sum_{n=0}^{m-1}
              \Bigl\|\frac{u_{\lambda, n+1}-u_{\lambda, n}}{h} 
                                   + \mu_{\lambda, n+1}-\mu_{\lambda, n}\Bigr\|_{V_{0}^{*}}^2 
+ \lambda h\sum_{n=0}^{m-1}
                     \Bigl\|\frac{u_{\lambda, n+1}-u_{\lambda, n}}{h}\Bigr\|_{H}^2 
\notag \\ 
&+ \ep\|u_{\lambda, m}\|_{V}^2 
   + \ep h^2 \sum_{n=0}^{m-1} 
                        \Bigl\|\frac{u_{\lambda, n+1}-u_{\lambda, n}}{h}\Bigr\|_{V}^2
   + \|u_{\lambda, m}\|_{L^4(\Omega)}^4  
\notag \\ 
&+ h\|\mu_{\lambda, m}\|_{H}^2  
   + h^3 \sum_{n=0}^{m-1}
            \Bigl\|\frac{\mu_{\lambda, n+1}-\mu_{\lambda, n}}{h}\Bigr\|_{H}^2 
   + h\sum_{n=0}^{m-1}\|\nabla\mu_{\lambda, n+1}\|_{H}^2  
\notag \\ 
&\leq C_{4}\ep h\sum_{j=0}^{m-1}\|u_{\lambda, j}\|_{V}^2 
        + C_{4} h\sum_{j=0}^{m-1}\|u_{\lambda, j}\|_{L^4(\Omega)}^4 
        + C_{4} 
\end{align*}
for all $h \in (0, h_{1})$ and $m=1, ..., N$. 
Therefore, owing to the discrete Gronwall lemma (see e.g., \cite[Prop.\ 2.2.1]{Jerome}),  
it holds that 
there exists a constant $C_{5}>0$ such that 
\begin{align*}
&h\sum_{n=0}^{m-1}
              \Bigl\|\frac{u_{\lambda, n+1}-u_{\lambda, n}}{h} 
                                   + \mu_{\lambda, n+1}-\mu_{\lambda, n}\Bigr\|_{V_{0}^{*}}^2 
+ \lambda h\sum_{n=0}^{m-1}
                     \Bigl\|\frac{u_{\lambda, n+1}-u_{\lambda, n}}{h}\Bigr\|_{H}^2 
\notag \\ 
&+ \ep\|u_{\lambda, m}\|_{V}^2 
   + \ep h^2 \sum_{n=0}^{m-1} 
                        \Bigl\|\frac{u_{\lambda, n+1}-u_{\lambda, n}}{h}\Bigr\|_{V}^2
   + \|u_{\lambda, m}\|_{L^4(\Omega)}^4  
\notag \\ 
&+ h\|\mu_{\lambda, m}\|_{H}^2  
   + h^3 \sum_{n=0}^{m-1}
            \Bigl\|\frac{\mu_{\lambda, n+1}-\mu_{\lambda, n}}{h}\Bigr\|_{H}^2 
   + h\sum_{n=0}^{m-1}\|\nabla\mu_{\lambda, n+1}\|_{H}^2  
\leq C_{5}
\end{align*}
for all $\ep \in (0, 1]$, $\lambda \in (0, \ep)$, $h \in (0, h_{1})$ 
and $m=1, ..., N$.  
\end{proof}
\begin{lem}\label{timedisces2}
Let $h_{1}$ be as in Lemma \ref{timedisces1}. 
Then there exists a constant $C>0$ depending on the data such that 
\begin{align*} 
\|\underline{u}_{h}\|_{L^{\infty}(0, T; L^4(\Omega))}^4 \leq C
\end{align*}
for all $\ep \in (0, 1]$, $\lambda \in (0, \ep)$, $h \in (0, h_{1})$.  
\end{lem}
\begin{proof}
We can obtain this lemma by Lemma \ref{timedisces1} and (A5).
\end{proof}
\begin{lem}\label{timedisces3}
Let $h_{1}$ be as in Lemma \ref{timedisces1}. 
Then there exists a constant $C>0$ depending on the data such that 
\begin{align*} 
\|(\hat{u}_{h})_{t}\|_{L^2(0, T; V^{*})}^2 \leq C
\end{align*}
for all $\ep \in (0, 1]$, $\lambda \in (0, \ep)$, $h \in (0, h_{1})$.  
\end{lem}
\begin{proof}
Since it follows from the first equation in \ref{Ph} that 
\begin{align*}
&\|(\hat{u}_{h})_{t}(t) + h(\hat{\mu}_{h})_{t}(t)\|_{V^{*}}^2 
\notag \\ 
&= \langle (\hat{u}_{h})_{t}(t) + h(\hat{\mu}_{h})_{t}(t),  
                      F^{-1}((\hat{u}_{h})_{t}(t) + h(\hat{\mu}_{h})_{t}(t)) \rangle_{V^{*}, V}
\notag \\ 
&= -\int_{\Omega}\nabla\overline{\mu}_{h}(t)\cdot
                               \nabla F^{-1}((\hat{u}_{h})_{t}(t) + h(\hat{\mu}_{h})_{t}(t)) 
\notag \\
&\,\quad + \eta\int_{\Omega}
                      \underline{u}_{h}(t)\nabla(1-\Delta)^{-1}\underline{u}_{h}(t)\cdot
                                       \nabla F^{-1}((\hat{u}_{h})_{t}(t) + h(\hat{\mu}_{h})_{t}(t)), 
\end{align*}
we see from the Young inequality, Lemmas \ref{timedisces1} and \ref{timedisces2} that 
there exists a constant $C_1 > 0$ such that 
\begin{align*}
\|(\hat{u}_{h})_{t} + h(\hat{\mu}_{h})_{t}\|_{L^2(0, T; V^{*})}^2 
\leq C_{1} 
\end{align*}
for all $\ep \in (0, 1]$, $\lambda \in (0, \ep)$, $h \in (0, h_{1})$. 
Then Lemma \ref{timedisces1} leads to Lemma \ref{timedisces3}. 
\end{proof}
\begin{lem}\label{timedisces4}
Let $h_{1}$ be as in Lemma \ref{timedisces1}. 
Then there exists a constant $C>0$ depending on the data such that 
\begin{align*} 
\ep^2\|\overline{u}_{h}\|_{L^2(0, T; W)}^2 \leq C
\end{align*}
for all $\ep \in (0, 1]$, $\lambda \in (0, \ep)$, $h \in (0, h_{1})$.  
\end{lem}
\begin{proof}
We have from the second equation in \ref{Plamn} that 
\begin{align}\label{c1}
&\ep^2 h \|\Delta u_{\lambda, n+1}\|_{H}^2 
= \ep^2 h (-\Delta u_{\lambda, n+1}, -\Delta u_{\lambda, n+1})_{H} 
\notag \\ 
&= \ep h (\mu_{\lambda, n+1}, -\Delta u_{\lambda, n+1})_{H} 
    - \frac{\ep\lambda}{2}
          (\|\nabla u_{\lambda, n+1}\|_{H}^2 - \|\nabla u_{\lambda, n}\|_{H}^2 
                  + \|\nabla u_{\lambda, n+1} - \nabla u_{\lambda, n}\|_{H}^2) 
\notag \\ 
&\,\quad - \ep h (\beta(u_{\lambda, n+1}), -\Delta u_{\lambda, n+1})_{H} 
   - \ep h (\pi_{\ep}(u_{\lambda, n+1})), -\Delta u_{\lambda, n+1})_{H} 
\notag \\
&\,\quad+ \ep h(f_{n+1}, -\Delta u_{\lambda, n+1})_{H}. 
\end{align}
Here Lemma \ref{PT} implies that 
\begin{align}\label{c2}
- \ep h (\beta(u_{\lambda, n+1}), -\Delta u_{\lambda, n+1})_{H} \leq 0. 
\end{align}
We deduce from the first equation in \ref{Plamn} that 
\begin{align}\label{c3}
&\ep h (\mu_{\lambda, n+1}, -\Delta u_{\lambda, n+1})_{H} 
= \ep h (-\Delta\mu_{\lambda, n+1}, u_{\lambda, n+1})_{H} 
\notag \\ 
&= -\frac{\ep}{2}
            (\|u_{\lambda, n+1}\|_{H}^2 - \|u_{\lambda, n}\|_{H}^2 
                             + \|u_{\lambda, n+1}-u_{\lambda, n}\|_{H}^2) 
    -\ep h^2 \Bigl(\frac{\mu_{\lambda, n+1}-\mu_{\lambda, n}}{h}, 
                                                                          u_{\lambda, n+1} \Bigr)_{H} 
\notag \\ 
  &\,\quad+ \eta\ep h\int_{\Omega}
                       u_{\lambda, n}\nabla(1-\Delta)^{-1}u_{\lambda, n}\cdot
                                                                             \nabla u_{\lambda, n+1}. 
\end{align}
We infer from 
the continuity of 
the embedding $W^{2, 4}(\Omega) \hookrightarrow W^{1, 4}(\Omega)$, 
standard elliptic regularity theory (\cite{F-1969}),   
Lemmas \ref{timedisces1} and \ref{timedisces2}   
that 
there exist constants $C_{1}, C_{2}, C_{3} > 0$ such that 
\begin{align}\label{c4}
&\eta\ep h\int_{\Omega}
                       u_{\lambda, n}\nabla(1-\Delta)^{-1}u_{\lambda, n}\cdot
                                                                             \nabla u_{\lambda, n+1}
\notag \\ 
&\leq \eta\ep h\|u_{\lambda, n}\|_{L^4(\Omega)}
                            \|\nabla(1-\Delta)^{-1}u_{\lambda, n}\|_{L^4(\Omega)}
                                                                        \|\nabla u_{\lambda, n+1}\|_{H} 
\notag \\ 
&\leq C_{1}\eta\ep h\|u_{\lambda, n}\|_{L^4(\Omega)}
                            \|(1-\Delta)^{-1}u_{\lambda, n}\|_{W^{2, 4}(\Omega)}
                                                                       \|\nabla u_{\lambda, n+1}\|_{H} 
\notag \\ 
&\leq C_{2}\eta\ep h\|u_{\lambda, n}\|_{L^4(\Omega)}^2\|\nabla u_{\lambda, n+1}\|_{H} 
\notag \\
&\leq C_{3}\eta h
\end{align}
for all $\ep \in (0, 1]$, $\lambda \in (0, \ep)$, $h \in (0, h_{1})$. 
The condition (A4) and Lemma \ref{timedisces1} yield that 
there exists a constant $C_{4}>0$ such that 
\begin{align}\label{c5}
- \ep h (\pi_{\ep}(u_{\lambda, n+1})), -\Delta u_{\lambda, n+1})_{H}
\leq c_{3}\ep h\|\nabla u_{\lambda, n+1}\|_{H}^2  
\leq C_{4}h
\end{align}
for all $\ep \in (0, 1]$, $\lambda \in (0, \ep)$, $h \in (0, h_{1})$. 
Therefore by \eqref{c1}-\eqref{c5}, the Young inequality, 
summing over $n=0, ..., m-1$ with $1 \leq m \leq N$, (A5)  
and Lemma \ref{timedisces1} 
there exists a constant $C_{5}>0$ such that 
\begin{align*}
\ep^2 \|\Delta\overline{u}_{h}\|_{L^2(0, T; H)}^2 \leq C_{5}
\end{align*}
for all $\ep \in (0, 1]$, $\lambda \in (0, \ep)$, $h \in (0, h_{1})$ 
and then 
we can obtain Lemma \ref{timedisces4} by Lemma \ref{timedisces1} 
and the elliptic regularity theory. 
\end{proof}
\begin{lem}\label{timedisces5}
Let $h_{1}$ be as in Lemma \ref{timedisces1}. 
Then there exists a constant $C>0$ depending on the data such that 
\begin{align*} 
\|\beta(\overline{u}_{h})\|_{L^2(0, T; L^1(\Omega))}^2 \leq C
\end{align*}
for all $\ep \in (0, 1]$, $\lambda \in (0, \ep)$, $h \in (0, h_{1})$.  
\end{lem}
\begin{proof}
Let $\tau>0$ and let $\beta_{\tau}$ be the Yosida approximation of $\beta$ 
on $\mathbb{R}$. 
Then there exist constants $C_{1}, C_{2} > 0$ such that 
\begin{align*}
\beta_{\tau}(r)(r-m_{0}) \geq C_{1}|\beta_{\tau}(r)| - C_{2}
\end{align*}
for all $r \in \mathbb{R}$ and all $\tau>0$ 
(see e.g., \cite[Section 5, p.\ 908]{GMS-2009}), 
where $m_{0}$ is as in (A5).    
Thus, since $\beta : \mathbb{R} \to \mathbb{R}$ is single-valued maximal monotone, 
it holds that $\beta_{\tau}(r) \to \beta(r)$ in $\mathbb{R}$ 
as $\tau\searrow0$ for all $r \in D(\beta)$ 
and then the inequality 
\begin{align}\label{e1}
\beta(r)(r-m_{0}) \geq C_{1}|\beta(r)| - C_{2}
\end{align}
holds for all $r \in D(\beta)$. 
The second equation in \ref{Plamn} leads to the identity 
\begin{align}\label{e2}
&h^{1/2}(\beta(u_{\lambda, n+1}), u_{\lambda, n+1}-m_{0})_{H} 
\notag \\ 
&= h^{1/2}(\mu_{\lambda, n+1}, u_{\lambda, n+1}+h\mu_{\lambda, n+1}-m_{0})_{H} 
    - h^{3/2}\|\mu_{\lambda, n+1}\|_{H}^2 
\notag \\
&\,\quad- \lambda h^{1/2}\Bigl(\frac{u_{\lambda, n+1}-u_{\lambda, n}}{h}, 
                                                                      u_{\lambda, n+1}-m_{0}\Bigr)_{H} 
- h^{1/2}\ep\|\nabla u_{\lambda, n+1}\|_{H}^2 
\notag \\
&\,\quad- h^{1/2}(\pi_{\ep}(u_{\lambda, n+1}), u_{\lambda, n+1}-m_{0})_{H} 
+ h^{1/2}(f_{n+1}, u_{\lambda, n+1}-m_{0})_{H}. 
\end{align}
Here, since it follows from \eqref{Plamnhozon} and (A5) that 
\begin{align*}%\label{e3}
\frac{1}{|\Omega|}\int_{\Omega}(u_{\lambda, j} + h\mu_{\lambda, j}) 
= \frac{1}{|\Omega|}\int_{\Omega}u_{0\ep} 
= m_{0}
\end{align*}
for $j = 0, ..., N$, 
we derive from the continuity of the embedding $H \hookrightarrow V_{0}^{*}$ that 
there exists a constant $C_{3}>0$ such that 
\begin{align}\label{e4}
&h^{1/2}(\mu_{\lambda, n+1}, u_{\lambda, n+1}+h\mu_{\lambda, n+1}-m_{0})_{H} 
\notag \\
&=h^{1/2}\Bigl\langle u_{\lambda, n+1}+h\mu_{\lambda, n+1}-m_{0}, 
             \mu_{\lambda, n+1} - \frac{1}{|\Omega|}\int_{\Omega}\mu_{\lambda, n+1}        
                                                                            \Bigr\rangle_{V_{0}^{*}, V_{0}}
\notag \\ 
&\leq h^{1/2}\|u_{\lambda, n+1}+h\mu_{\lambda, n+1}-m_{0}\|_{V_{0}^{*}}
                                                                  \|\nabla\mu_{\lambda, n+1}\|_{H} 
\notag \\ 
&\leq C_{3}h^{1/2}\|u_{\lambda, n+1}+h\mu_{\lambda, n+1}-m_{0}\|_{H}
                                                                  \|\nabla\mu_{\lambda, n+1}\|_{H} 
\end{align}
for all $\ep \in (0, 1]$, $\lambda \in (0, \ep)$, $h \in (0, h_{1})$. 
Therefore we can obtain Lemma \ref{timedisces5} 
by \eqref{e1}-\eqref{e4} and Lemma \ref{timedisces1}. 
\end{proof}
\begin{lem}\label{timedisces6}
Let $h_{1}$ be as in Lemma \ref{timedisces1}. 
Then there exists a constant $C>0$ depending on the data such that 
\begin{align*} 
\|\overline{\mu}_{h}\|_{L^2(0, T; V)}^2 
+ \|\beta(\overline{u}_{h})\|_{L^2(0, T; H)}^2 \leq C
\end{align*}
for all $\ep \in (0, 1]$, $\lambda \in (0, \ep)$, $h \in (0, h_{1})$.  
\end{lem}
\begin{proof}
We see from the second equation in \ref{Ph} that 
\begin{align*}
\int_{\Omega}\overline{\mu}_{h} 
= \lambda\int_{\Omega}(\hat{u}_{h})_{t} 
   + \int_{\Omega}\beta(\overline{u}_{h}) 
   + \int_{\Omega}\pi_{\ep}(\overline{u}_{h})
   - \int_{\Omega}\overline{f}_{h}.  
\end{align*}
Thus Lemmas \ref{timedisces1} and \ref{timedisces5} mean that 
there exists a constant $C_{1}>0$ such that 
\begin{align*}
\int_{0}^{T}\left|\int_{\Omega}\overline{\mu}_{h}\right|^2 \leq C_{1}
\end{align*}
for all $\ep \in (0, 1]$, $\lambda \in (0, \ep)$, $h \in (0, h_{1})$, 
and hence we have from Lemma \ref{timedisces1} and 
the Poincar\'e--Wirtinger inequality 
that there exists a constant $C_{2}>0$ such that 
\begin{align}\label{f1}
\|\overline{\mu}_{h}\|_{L^2(0, T; V)}^2 \leq C_{2}
\end{align}
for all $\ep \in (0, 1]$, $\lambda \in (0, \ep)$, $h \in (0, h_{1})$. 
Therefore the second equation in \ref{Ph}, 
Lemmas \ref{timedisces1} and \ref{timedisces4},  
and \eqref{f1} 
that there exists a constant $C_{3}>0$ such that 
\begin{align*}
\|\beta(\overline{u}_{h})\|_{L^2(0, T; H)}^2 \leq C_{3}
\end{align*}
for all $\ep \in (0, 1]$, $\lambda \in (0, \ep)$, $h \in (0, h_{1})$. 
\end{proof}
\begin{lem}\label{timedisces7}
Let $h_{1}$ be as in Lemma \ref{timedisces1}. 
Then there exists a constant $C>0$ depending on the data such that 
\begin{align*} 
&\lambda\|\hat{u}_{h}\|_{H^1(0, T; H)}^2 
+ \ep\|\hat{u}_{h}\|_{L^{\infty}(0, T; V)}^2 
+ \|\hat{u}_{h}\|_{L^{\infty}(0, T; L^4(\Omega))}^4 
+ \|\hat{u}_{h}\|_{H^1(0, T; V^{*})}^2
\\[2mm] 
&+ h\|\hat{\mu}_{h}\|_{L^{\infty}(0, T; H)}^2 
+ h^2\|\hat{\mu}_{h}\|_{H^1(0, T; H)}^2 
\leq C  
\end{align*}
for all $\ep \in (0, 1]$, $\lambda \in (0, \ep)$, $h \in (0, h_{1})$.  
\end{lem}
\begin{proof}
This lemma holds by \eqref{rem1}-\eqref{rem3}, 
Lemmas \ref{timedisces1} and \ref{timedisces3}.   
\end{proof}
\medskip

Now we set the operator $\Phi : L^4(\Omega) \to V^{*}$ as 
\begin{align}\label{PhiL4toVstar}
&\langle \Phi u, z \rangle_{V^{*}, V} 
:= -\eta\int_{\Omega} u\nabla(1-\Delta)^{-1}u\cdot\nabla z   
\quad \mbox{for}\ u \in L^4(\Omega),\ z \in V.  
\end{align}
\begin{prlem2.2}
Let $\ep \in (0, 1]$ and let $\lambda \in (0, \ep)$. 
Then 
by Lemmas \ref{timedisces1}, \ref{timedisces2}, \ref{timedisces4}, \ref{timedisces6}, 
\ref{timedisces7}, 
recalling \eqref{rem4}, \eqref{rem5} 
and the compactness of the embedding $V \hookrightarrow H$ 
there exist some functions 
$u_{\ep, \lambda}$, $\mu_{\ep, \lambda}$, $\xi_{\ep, \lambda}$
such that 
\begin{align*}
&u_{\ep, \lambda} \in H^1(0, T; H) \cap L^{\infty}(0, T; V) \cap L^2(0, T; W),\ 
\mu_{\ep, \lambda} \in L^2(0, T; V), 
\\
&\xi_{\ep, \lambda} \in L^2(0, T; H)
\end{align*}
and 
\begin{align}
&\hat{u}_{h} \to u_{\ep, \lambda} 
\quad \mbox{weakly$^{*}$ in}\ H^1(0, T; H) \cap L^{\infty}(0, T; V), \label{1weak1} 
\\
&\hat{u}_{h} \to u_{\ep, \lambda} 
\quad \mbox{strongly in}\ C([0, T]; H), \label{1strong1} 
\\ 
&h\hat{\mu}_{h} \to 0 
\quad \mbox{strongly in}\ L^{\infty}(0, T; H), \notag 
\\ 
&h\hat{\mu}_{h} \to 0 
\quad \mbox{weakly in}\ H^1(0, T; H), \label{1weak2} 
\\ 
&\overline{\mu}_{h} \to \mu_{\ep, \lambda} 
\quad \mbox{weakly in}\ L^2(0, T; V), \label{1weak3}
\\ 
&\beta(\overline{u}_{h}) \to \xi_{\ep, \lambda} 
\quad \mbox{weakly in}\ L^2(0, T; H), \label{1weak4} 
\\ 
&\overline{u}_{h} \to u_{\ep, \lambda} 
\quad \mbox{weakly$^{*}$ in}\ L^{\infty}(0, T; L^4(\Omega)), \notag 
\\
&\overline{u}_{h} \to u_{\ep, \lambda} 
\quad \mbox{weakly in}\ L^2(0, T; W), \label{1weak5} 
\\
&\underline{u}_{h} \to u_{\ep, \lambda} 
\quad \mbox{weakly$^{*}$ in}\ L^{\infty}(0, T; L^4(\Omega)) \label{1weak6} 
\end{align}
as $h=h_{j}\searrow0$. 
We can verify \eqref{de9} by \eqref{1strong1}. 
Now we check \eqref{de8}. 
It follows from \eqref{rem4}, Lemma \ref{timedisces1} and \eqref{1strong1} 
that 
\begin{align}\label{g1}
\|\overline{u}_{h}-u_{\ep, \lambda}\|_{L^2(0, T; H)} 
&\leq \|\overline{u}_{h}-\hat{u}_{h}\|_{L^2(0, T; H)} 
        + \|\hat{u}_{h}-u_{\ep, \lambda}\|_{L^2(0, T; H)} 
\notag \\ 
&= \frac{h}{\sqrt{3}}\|(\hat{u}_{h})_{t}\|_{L^2(0, T; H)} 
         + \|\hat{u}_{h}-u_{\ep, \lambda}\|_{L^2(0, T; H)} 
\notag \\ 
&\to 0
\end{align}
as $h=h_{j}\searrow0$. 
Therefore we see from \eqref{1weak4} and \eqref{g1} that 
\begin{align*}
\int_{0}^{T}(\beta(\overline{u}_{h}(t)), \overline{u}_{h}(t))_{H}\,dt 
\to \int_{0}^{T} (\xi_{\ep, \lambda}(t), u_{\ep, \lambda}(t))_{H}\,dt
\end{align*} 
as $h=h_{j}\searrow0$, whence it holds that  
\begin{align}\label{g2}
\xi_{\ep, \lambda} = \beta(u_{\ep, \lambda}) 
\quad \mbox{a.e.\ on}\ \Omega\times(0, T) 
\end{align}
(see e.g., \cite[Lemma 1.3, p.\ 42]{Barbu1}).  
Thus we can obtain \eqref{de8} by the second equation in \ref{Ph}, 
\eqref{1weak1}, \eqref{1weak3}-\eqref{1weak5}, \eqref{g2},  
the Lipschitz continuity of $\pi_{\ep}$, \eqref{g1}, 
and by observing that $\overline{f}_{h} \to f$ 
strongly in $L^2(0, T; V)$ as $h\searrow0$ (cf.\ \cite[Section 5]{CK1}).       

Next we prove \eqref{de7}. 
Owing to standard elliptic regularity theory (\cite{F-1969}) and Lemma \ref{timedisces7}, 
there exists a constant $C_{1} > 0$ such that 
\begin{align*}
&\|F^{-1}(\hat{u}_{h})_{t}\|_{L^2(0, T; V)} 
= \|(\hat{u}_{h})_{t}\|_{L^2(0, T; V^{*})} \leq C_{1}, 
\\ 
&\|F^{-1}\hat{u}_{h}\|_{L^{\infty}(0, T; W^{2, 4}(\Omega))} 
= \|(-\Delta + 1)^{-1}\hat{u}_{h}\|_{L^{\infty}(0, T; W^{2, 4}(\Omega))} \leq C_{1}
\end{align*}  
for all $h \in (0, h_{1})$ 
and then we have from 
$W^{2, 4}(\Omega) \hookrightarrow W^{1, 4}(\Omega) \hookrightarrow V$  
with compactness embedding $W^{2, 4}(\Omega) \hookrightarrow W^{1, 4}(\Omega)$ 
and Lemma \ref{ALlem} that 
\begin{align*}
F^{-1}\hat{u}_{h} \to F^{-1} u_{\ep, \lambda} 
\quad \mbox{strongly in}\ C([0, T]; W^{1, 4}(\Omega)) 
\end{align*}
as $h=h_{j}\searrow0$, 
which yields that 
\begin{align}\label{g3}
(1-\Delta)^{-1}\hat{u}_{h} \to (1-\Delta)^{-1} u_{\ep, \lambda} 
\quad \mbox{strongly in}\ L^2(0, T; W^{1, 4}(\Omega)) 
\end{align}
as $h=h_{j}\searrow0$. 
Thus we derive from \eqref{rem5}, 
the continuity of 
the embedding $W \hookrightarrow W^{1, 4}(\Omega)$, 
the elliptic regularity theory, 
\eqref{rem4}, Lemma \ref{timedisces1} and \eqref{g3}  
that 
\begin{align*}
&\|(1-\Delta)^{-1}\underline{u}_{h} 
                          - (1-\Delta)^{-1}u_{\ep, \lambda}\|_{L^2(0, T; W^{1, 4}(\Omega))} 
\notag \\ 
&\leq h\|(1-\Delta)^{-1}({u}_{h})_{t}\|_{L^2(0, T; W^{1, 4}(\Omega))} 
         + \|(1-\Delta)^{-1}(\overline{u}_{h} - \hat{u}_{h})\|_{L^2(0, T; W^{1, 4}(\Omega))} 
\notag \\
   &\,\quad+ \|(1-\Delta)^{-1}\hat{u}_{h} 
                          - (1-\Delta)^{-1}u_{\ep, \lambda}\|_{L^2(0, T; W^{1, 4}(\Omega))}  
\notag \\ 
&\leq C_{2}h\|(1-\Delta)^{-1}({u}_{h})_{t}\|_{L^2(0, T; W)} 
         + C_{2}\|(1-\Delta)^{-1}(\overline{u}_{h} - \hat{u}_{h})\|_{L^2(0, T; W)} 
\notag \\ 
    &\,\quad+ \|(1-\Delta)^{-1}\hat{u}_{h} 
                          - (1-\Delta)^{-1}u_{\ep, \lambda}\|_{L^2(0, T; W^{1, 4}(\Omega))}  
\notag \\ 
&\leq C_{3}h\|(\hat{u}_{h})_{t}\|_{L^2(0, T; H)}
         + \|(1-\Delta)^{-1}\hat{u}_{h} 
                          - (1-\Delta)^{-1}u_{\ep, \lambda}\|_{L^2(0, T; W^{1, 4}(\Omega))}  
\notag \\ 
&\to 0 
\end{align*}
as $h=h_{j}\searrow0$, where $C_{2}, C_{3} > 0$ are some constants.    
Hence we infer that 
\begin{align}\label{g4}
\nabla(1-\Delta)^{-1}\underline{u}_{h} \to \nabla(1-\Delta)^{-1}u_{\ep, \lambda} 
\quad \mbox{strongly in}\ L^2(0, T; L^4(\Omega))
\end{align}
as $h=h_{j}\searrow0$. 
Moreover, Lemma \ref{timedisces2}, the convergences \eqref{1weak6} and \eqref{g4} 
imply that
\begin{align}\label{g5}
&\left|\int_{0}^{T}\int_{\Omega} 
           (\underline{u}_{h}\nabla(1-\Delta)^{-1}\underline{u}_{h}
             -u_{\ep, \lambda}\nabla(1-\Delta)^{-1}u_{\ep, \lambda})\cdot\nabla\psi \right| 
\notag \\ 
&\leq \left|\int_{0}^{T}\int_{\Omega} 
           (\underline{u}_{h}-u_{\ep, \lambda})\nabla(1-\Delta)^{-1}u_{\ep, \lambda}
                                                                                       \cdot\nabla\psi \right| 
\notag \\ 
 &\,\quad+ \left|\int_{0}^{T}\int_{\Omega} 
                \underline{u}_{h}\nabla(1-\Delta)^{-1}(\underline{u}_{h}-u_{\ep, \lambda})
                                                                                       \cdot\nabla\psi \right| 
\notag \\ 
&\leq \left|\int_{0}^{T}
                  \langle 
                      \underline{u}_{h}(t)-u_{\ep, \lambda}(t), 
                           \nabla(1-\Delta)^{-1}u_{\ep, \lambda}(t)\cdot\nabla\psi(t) 
                                                \rangle_{L^4(\Omega), L^{4/3}(\Omega)}\,dt \right| 
\notag \\ 
&\,\quad+ C_{4}\|\nabla(1-\Delta)^{-1}
                             (\underline{u}_{h} - u_{\ep, \lambda})\|_{L^2(0, T; L^4(\Omega))}
                               \|\nabla\psi\|_{L^2(0, T; H)} 
\notag \\ 
&\to 0
\end{align}
as $h=h_{j}\searrow0$, where $C_{4}>0$ is some constant.  
Therefore it follows from the first equation in \ref{Ph}, 
\eqref{1weak1}, \eqref{1weak2}, \eqref{1weak3}, \eqref{g5} and \eqref{PhiL4toVstar} 
that 
\begin{align*}
(u_{\ep, \lambda})_{t} + (F - I)\mu_{\ep, \lambda} + \Phi(u_{\ep, \lambda}) = 0 
\quad \mbox{in}\ L^2(0, T; V^{*}),  
\end{align*}
which leads to \eqref{de7}. 
\qed
\end{prlem2.2}

\vspace{10pt}
 
%%==============================================================%%
%%==============                                  ==============%%
%%======                      Section5                    ======%%
%%====                                                      ====%%
%%==                                                          ==%%
%%====                                                      ====%%
%%======                                                  ======%%
%%==============                                  ==============%%
%%==============================================================%%
\section{Proof of main theorems}\label{Sec5}
In this section we will prove Theorems \ref{maintheorem1} and \ref{maintheorem2}.  
\begin{lem}\label{aboutestiforPlamep}
There exists a constant $M>0$, depending only on the data, 
such that 
     \begin{align}
     & \lambda\int_{0}^{t}\|u_{\ep, \lambda}'(s)\|_{H}^2\,ds 
     + \ep \|u_{\ep, \lambda}(t)\|_{V}^2 
     + \|u_{\ep, \lambda}(t)\|_{L^4(\Omega)}^4 \leq M, \label{lamepes1} 
     \\[2mm] 
     &\int_{0}^{t}\|u_{\ep, \lambda}'(s)\|_{V^*}^2\,ds \leq M,  \label{lamepes2} 
     \\[2mm] 
     &\ep^2\int_{0}^{t}\|u_{\ep, \lambda}(s)\|_{W}^2\,ds \leq M, \label{lamepes3} 
     \\[2mm] 
     &\int_{0}^{t}\|\mu_{\ep, \lambda}(s)\|_{V}^2\,ds 
            + \int_{0}^{t}\|\beta(u_{\ep, \lambda}(s))\|_{H}^2\,ds \leq M \label{lamepes4}
     \end{align}
for all $t\in[0, T]$, $\ep\in(0, 1]$, $\lambda\in(0, \ep)$. 
\end{lem}
\begin{proof}
This lemma holds 
by Lemmas \ref{timedisces1}, \ref{timedisces3}, \ref{timedisces4} and \ref{timedisces6}. 
\end{proof}
\begin{prth1.1}
Let $\ep \in (0, 1]$. 
Then we have from Lemma \ref{aboutestiforPlamep},  
the compactness of the embedding $V \hookrightarrow H$ 
that 
there exist some functions $u_{\ep}$, $\mu_{\ep}$, $\xi_{\ep}$ such that 
\begin{align*}
&u_{\ep} \in H^1(0, T; V^{*}) \cap L^{\infty}(0, T; V) \cap L^2(0, T; W),\ 
\mu_{\ep} \in L^2(0, T; V), 
\\
&\xi_{\ep} \in L^2(0, T; H) 
\end{align*}
and   
\begin{align}
&\lambda u_{\ep, \lambda} \to 0 
\quad \mbox{strongly in}\ H^1(0, T; H), \label{2strong1} 
\\ 
&u_{\ep, \lambda} \to u_{\ep} 
\quad \mbox{weakly$^{*}$ in}\ 
                H^1(0, T; V^{*}) \cap L^{\infty}(0, T; V) \cap L^2(0, T; W), \label{2weak1} 
\\ 
&u_{\ep, \lambda} \to u_{\ep} 
\quad \mbox{strongly in}\ 
                        C([0, T]; H),  \label{2strong2} 
\\ 
&\mu_{\ep, \lambda} \to \mu_{\ep} 
\quad \mbox{weakly in}\ L^2(0, T; V), \label{2weak2} 
\\ 
&\beta(u_{\ep, \lambda}) \to \xi_{\ep} 
\quad \mbox{weakly in}\ L^2(0, T; H) \label{2weak3}
\end{align}
as $\lambda=\lambda_{j}\searrow0$. 
We can check \eqref{de6} by \eqref{2strong2}. 
Now we show \eqref{de5}. 
The convergences \eqref{2weak3} and \eqref{2strong2} yield that 
\begin{align*}
\int_{0}^{T}(\beta(u_{\ep, \lambda}(t)), u_{\ep, \lambda}(t))_{H}\,dt 
\to \int_{0}^{T}(\xi_{\ep}(t), u_{\ep}(t))_{H}\,dt 
\end{align*}
as $\lambda=\lambda_{j}\searrow0$,  
which means that 
\begin{align}\label{h1}
\xi_{\ep} = \beta(u_{\ep}) \quad \mbox{a.e.\ on}\ \Omega\times(0, T)  
\end{align}
(see e.g., \cite[Lemma 1.3, p.\ 42]{Barbu1}).   
Thus we can confirm that \eqref{de5} holds  
by \eqref{de8}, \eqref{2weak2}, \eqref{2strong1}, \eqref{2weak1}, \eqref{2weak3}, 
\eqref{h1}, 
the Lipschitz continuity of $\pi_{\ep}$, and \eqref{2strong2}.  

Next we verify \eqref{de4}. 
It follows from standard elliptic regularity theory (\cite{F-1969}) 
and Lemma \ref{aboutestiforPlamep} that 
there exists a constant $C_{1}>0$ such that 
\begin{align*}
&\|F^{-1}(u_{\ep, \lambda})_{t}\|_{L^2(0, T; V)} 
= \|(u_{\ep, \lambda})_{t}\|_{L^2(0, T; V^{*})} \leq C_{1}, 
\\ 
&\|F^{-1}u_{\ep, \lambda}\|_{L^{\infty}(0, T; W^{2, 4}(\Omega))} 
= \|(-\Delta + 1)^{-1}u_{\ep, \lambda}\|_{L^{\infty}(0, T; W^{2, 4}(\Omega))} \leq C_{1}
\end{align*}  
for all $\lambda \in (0, \ep)$. 
Thus we deduce from 
$W^{2, 4}(\Omega) \hookrightarrow W^{1, 4}(\Omega) \hookrightarrow V$  
with compactness embedding $W^{2, 4}(\Omega) \hookrightarrow W^{1, 4}(\Omega)$ 
and Lemma \ref{ALlem} that 
\begin{align*}
F^{-1}u_{\ep, \lambda} \to F^{-1} u_{\ep} 
\quad \mbox{strongly in}\ C([0, T]; W^{1, 4}(\Omega)) 
\end{align*}
as $\lambda=\lambda_{j}\searrow0$, 
which leads to the convergence  
\begin{align*}
(1-\Delta)^{-1}u_{\ep, \lambda} \to (1-\Delta)^{-1} u_{\ep} 
\quad \mbox{strongly in}\ L^4(0, T; W^{1, 4}(\Omega)) 
\end{align*}
as $\lambda=\lambda_{j}\searrow0$. 
Then it holds that 
\begin{align}\label{h2}
\nabla(1-\Delta)^{-1}u_{\ep, \lambda}\cdot\nabla\psi 
\to \nabla(1-\Delta)^{-1} u_{\ep}\cdot\nabla\psi  
\quad \mbox{strongly in}\ L^{4/3}(0, T; L^{4/3}(\Omega)) 
\end{align}
as $\lambda=\lambda_{j}\searrow0$ for all $\psi \in L^2(0, T; V)$. 
Therefore we see 
from \eqref{de7}, \eqref{2weak1}, \eqref{2weak2} and \eqref{h2}
that 
\begin{align*}
(u_{\ep})_{t} + (F - I)\mu_{\ep} + \Phi(u_{\ep}) = 0 
\quad \mbox{in}\ L^2(0, T; V^{*}),   
\end{align*}
where $\Phi$ is as in \eqref{PhiL4toVstar},  
and then we can obtain \eqref{de4}. 
\qed
\end{prth1.1}
\begin{prth1.2}
By \eqref{epes1}-\eqref{epes4} 
and the compactness of the embedding $H \hookrightarrow V^{*}$  
there exist some functions $u$, $\mu$, $\xi$ such that 
\begin{align*}
&u \in H^1(0, T; V^{*}) \cap L^{\infty}(0, T; L^4(\Omega)),\ 
\mu \in L^2(0, T; V), 
\\ 
&\xi \in L^2(0, T; H) 
\end{align*}
and 
\begin{align}
&u_{\ep} \to u 
\quad \mbox{weakly$^{*}$ in}\
     H^1(0, T; V^{*}) \cap L^{\infty}(0, T; L^4(\Omega)), \label{3weak1} 
\\ 
&u_{\ep} \to u 
\quad \mbox{strongly in}\ C([0, T]; V^{*}), \label{3strong1} 
\\ 
&\ep u_{\ep} \to 0 
\quad \mbox{strongly in}\ L^{\infty}(0, T; V), \notag 
\\ 
&\ep u_{\ep} \to 0 
\quad \mbox{weakly in}\ L^20, T; W), \label{3weak2} 
\\ 
&\mu_{\ep} \to \mu 
\quad \mbox{weakly in}\ L^2(0, T; V), \label{3weak3} 
\\ 
&\beta(u_{\ep}) \to \xi 
\quad \mbox{weakly in}\ L^2(0, T; H) \label{3weak4} 
\end{align}
as $\ep=\ep_{j}\searrow0$. 
We can prove \eqref{de3} by \eqref{3strong1} and (A5). 
Also we can show \eqref{de1} by \eqref{de4}, \eqref{3weak1}, \eqref{3weak3}, 
\eqref{epes1}, \eqref{epes2}, 
$W^{2, 4}(\Omega) \hookrightarrow W^{1, 4}(\Omega) \hookrightarrow V$  
with compactness embedding $W^{2, 4}(\Omega) \hookrightarrow W^{1, 4}(\Omega)$ 
and Lemma \ref{ALlem}
in the same way as in the proof of Theorem \ref{maintheorem1}.
Now we check \eqref{de2}. 
We have from (A4) and \eqref{epes1} that 
\begin{align}\label{i1}
\|\pi_{\ep}(u_{\ep})\|_{L^{\infty}(0, T; L^4(\Omega))} 
&\leq |\Omega|^{1/4}|\pi_{\ep}(0)| 
         + \|\pi_{\ep}(u_{\ep})-\pi_{\ep}(0)\|_{L^{\infty}(0, T; L^4(\Omega))} 
\notag \\ 
&\leq |\Omega|^{1/4}|\pi_{\ep}(0)| 
         + \|\pi_{\ep}'\|_{L^{\infty}(\mathbb{R})}\|u_{\ep}\|_{L^{\infty}(0, T; L^4(\Omega))} 
\notag \\ 
&\leq C_{1}\ep 
\notag \\ 
&\to 0
\end{align}
as $\ep\searrow0$, 
where $C_{1}>0$ is some constant.  
Hence we derive from \eqref{de5}, \eqref{3weak3}, \eqref{3weak2}, \eqref{3weak4} 
and \eqref{i1} 
that 
\begin{align}\label{i2}
\mu = \xi - f \quad \mbox{a.e.\ on}\ \Omega\times(0, T). 
\end{align}
Moreover, we infer from \eqref{de5}, \eqref{3strong1}, \eqref{3weak3}, 
\eqref{i1}, \eqref{3weak1} and \eqref{i2} that   
\begin{align*}
&\int_{0}^{T}(\beta(u_{\ep}(t)), u_{\ep}(t))_{H}\,dt 
=\int_{0}^{T}(\mu_{\ep}(t) + \ep\Delta u_{\ep}(t) - \pi_{\ep}(u_{\ep}(t)) + f(t), 
                                                                                             u_{\ep}(t))_{H}\,dt 
\notag \\ 
&= \int_{0}^{T}\langle u_{\ep}(t), \mu_{\ep}(t) \rangle_{V^{*}, V}\,dt 
     - \ep\int_{0}^{T} \|\nabla u_{\ep}(t)\|_{H}^2\,dt 
     - \int_{0}^{T}(\pi_{\ep}(u_{\ep}(t)), u_{\ep}(t))_{H}\,dt 
\notag \\ 
&\,\quad + \int_{0}^{T}(f(t), u_{\ep}(t))_{H}\,dt 
\notag \\ 
&\leq \int_{0}^{T}\langle u_{\ep}(t), \mu_{\ep}(t) \rangle_{V^{*}, V}\,dt 
     - \int_{0}^{T}(\pi_{\ep}(u_{\ep}(t)), u_{\ep}(t))_{H}\,dt 
     + \int_{0}^{T}(f(t), u_{\ep}(t))_{H}\,dt 
\notag \\ 
&\to \int_{0}^{T}\langle u(t), \mu(t) \rangle_{V^{*}, V}\,dt 
        + \int_{0}^{T}(f(t), u(t))_{H}\,dt 
= \int_{0}^{T}(\xi(t), u(t))_{H}\,dt 
\end{align*} 
as $\ep=\ep_{j}\searrow0$. 
Thus it holds that 
\begin{align}\label{i3}
\xi = \beta(u) \quad \mbox{a.e.\ on}\ \Omega\times(0, T) 
\end{align}
(see e.g., \cite[Lemma 1.3, p.\ 42]{Barbu1}).  
Therefore we can verify \eqref{de2} 
by \eqref{i2} and \eqref{i3}. 
\qed
\end{prth1.2}

\section*{Acknowledgments}
The author is supported by JSPS Research Fellowships 
for Young Scientists (No.\ 18J21006).   
%
%
%
%%==============================================================%%
%%==============                                  ==============%%
%%======                                                  ======%%
%%====                                                      ====%%
%%==                         Reference                        ==%%
%%====                                                      ====%%
%%======                                                  ======%%
%%==============                                  ==============%%
%%==============================================================%%

\end{document}